\documentclass[12pt]{amsart}

\usepackage{amsmath,amsfonts,amsthm,amssymb,graphicx,enumerate,tikz}
\usepackage[all,2cell,ps]{xy}
\usepackage[title]{appendix}
\theoremstyle{plain}
\newtheorem{thm}{Theorem}[section]
\newtheorem{lem}[thm]{Lemma}
\newtheorem{prop}[thm]{Proposition}
\newtheorem{cor}[thm]{Corollary}
\newtheorem{conj}[thm]{Conjecture}

\theoremstyle{definition}
\newtheorem{rem}[thm]{Remark}
\newtheorem{rems}[thm]{Remarks}

\newtheorem*{exs}{Examples}
\newcommand{\C}{\mathbb{C}}
\newcommand{\Q}{\mathbb{Q}}
\newcommand{\Z}{\mathbb{Z}}
\newcommand{\F}{\mathbb{F}}
\newcommand{\R}{\mathbb{R}}

\newcommand{\Del}{\Delta}
\newcommand{\Gam}{\Gamma}
\newcommand{\gam}{\gamma}

\DeclareMathOperator{\SL}{\mathrm{SL}}
\DeclareMathOperator{\SU}{\mathrm{SU}}
\DeclareMathOperator{\PSL}{\mathrm{PSL}}
\DeclareMathOperator{\M}{\mathrm{M}}
\DeclareMathOperator{\GL}{\mathrm{GL}}
\DeclareMathOperator{\PGL}{\mathrm{PGL}}

\DeclareMathOperator{\Isom}{\mathrm{Isom}}
\DeclareMathOperator{\Spec}{\mathrm{Spec}}
\DeclareMathOperator{\tr}{\mathrm{tr}}
\newcommand{\Hthree}{\mathbf{H}^3}
\DeclareMathOperator{\Br}{\mathrm{Br}}
\DeclareMathOperator{\Gal}{\mathrm{Gal}}

\DeclareMathOperator{\ord}{\mathrm{ord}}
\DeclareMathOperator{\End}{\mathrm{End}}

\DeclareFontFamily{U}{wncy}{}
\DeclareFontShape{U}{wncy}{m}{n}{<->wncyr10}{}
\DeclareSymbolFont{mcy}{U}{wncy}{m}{n}
\DeclareMathSymbol{\Sha}{\mathord}{mcy}{"58} 

\newcommand{\wt}{\widetilde}
\newcommand{\ssm}{\smallsetminus}
\title{Azumaya algebras and canonical components}
\author{Ted Chinburg}
\address{University of Pennsylvania}
\email{ted@math.upenn.edu}

\author{Alan W.\ Reid}
\address{Rice University}
\email{alan.reid@rice.edu}

\author{Matthew Stover}
\address{Temple University}
\email{mstover@temple.edu}
\date{\today}

\begin{document}

\begin{abstract}
Let $M$ be a compact 3-manifold and $\Gamma=\pi_1(M)$. Work of Thurston and Culler--Shalen established the $\SL_2(\C)$ character variety $X(\Gamma)$ as fundamental tool in the study of the geometry and topology of $M$. This is particularly the case when $M$ is the exterior of a hyperbolic knot $K$ in $S^3$. The main goals of this paper are to bring to bear tools from algebraic and arithmetic geometry to understand algebraic and number theoretic properties of the so-called canonical component of $X(\Gamma)$, as well as distinguished points on the canonical component, when $\Gamma$ is a knot group. In particular, we study how the theory of quaternion Azumaya algebras can be used to obtain algebraic and arithmetic information about Dehn surgeries, and perhaps of most interest, to construct new knot invariants that lie in the Brauer groups of curves over number fields.
\end{abstract}

\maketitle
%%%%%%%%%%%%%%%%%%%%

%%%%%%%%%%%%%%%%%%%%
\section{Introduction}\label{sec:Intro}
%%%%%%%%%%%%%%%%%%%%

Let $\Gamma$ be a finitely generated group and let $X(\Gam)$ denote the $\SL_2(\C)$ character variety of $\Gam$ (see \S \ref{sec:CharVars}). When $\Gamma$ is the fundamental group of a compact 3-manifold $M$, seminal work of Thurston and Culler--Shalen established $X(\Gamma)$ as a powerful tool in the study of the geometry and topology of $M$. The aim of this paper is to bring to bear tools from algebraic and arithmetic geometry to understand algebraic and number theoretic properties of certain components of $X(\Gamma)$ for arbitrary finitely generated $\Gamma$, as well as distinguished points on these components. In particular, we study how the theory of quaternion Azumaya algebras (see \S \ref{sec:Azumaya}) can be used to obtain algebraic and arithmetic information about Dehn surgeries on $1$-cusped hyperbolic $3$-manifolds. This leads, in particular, to new knot invariants that lie in the Brauer groups of curves over number fields.

Our approach is two-fold. First, we apply results on Azumaya algebras, classical and recent, to prove results about invariants of Dehn surgeries on hyperbolic knots. Second, we show how conditions from $3$-manifold topology, for example arithmetic properties of the Alexander polynomial of a knot, help prove the existence of an Azumaya algebra over certain curves defined over a number field. In the remainder of the introduction we expand on this theme and state a number of the results we will prove.

Let $\Gamma$ be as above, and suppose that $\rho : \Gam \to \SL_2(\C)$ is an absolutely irreducible representation. If $\chi_\rho$ is the character of $\rho$, let $k_\rho$ be the field generated over $\Q$ by the values of $\chi_\rho$. Then the $k_\rho$-span of $\rho(\Gam)$ defines a $k_\rho$-quaternion subalgebra $A_\rho \subset \M_2(\C)$ (cf.\ \cite[Thm.\ 3.2.1]{MaclachlanReid}). When $\rho(\Gam)$ is a discrete subgroup of $\SL_2(\C)$ of finite co-volume, the field $k_\rho$ and the algebra $A_\rho$ are important geometric and topological invariants of the hyperbolic $3$-manifold $M_\rho = \Hthree / \rho(\Gam)$ that are closely related to the lengths of closed geodesics and the spectrum of the Laplace--Beltrami operator on $M_\rho$ (see \cite{CHLR}).

Suppose that $\Gam = \pi_1(S^3 \ssm K)$ is the fundamental group of a hyperbolic knot complement and $\rho : \Gam \to \SL_2(\C)$ is the discrete representation associated with a hyperbolic Dehn surgery on $S^3 \ssm K$. One can then ask if and how the invariants $k_\rho$ and $A_\rho$ of $M_\rho = \Hthree / \rho(\Gam)$ depend on $K$. Our work was partially motivated by giving a theoretical explanation for the following examples.

%%%%%%%%%%%%%%%%%%%%
\begin{exs}
Using the program Snap \cite{snap}, one can determine the so-called \emph{invariants} of the algebra $A_\rho$, where $\rho : \Gam \to \SL_2(\C)$ is the representation associated with a hyperbolic Dehn surgery on $S^3 \ssm K$ with sufficiently small surgery coefficient. The non-trivial invariants are a finite list of real and finite places of the trace field $k_\rho$, which is also called the \emph{ramification set} for $A_\rho$. For the knot $4_1$, the figure-eight knot, every finite place that appears in the ramification set has residue characteristic $2$, i.e., the associated prime ideal of the ring of integers $\mathcal{O}_{k_\rho}$ of $k_\rho$ divides $2 \mathcal{O}_{k_\rho}$. For other knots, the invariants behave much more wildly: for the knot $5_2$ one sees invariants with residue characteristics including $5$, $13$, and $181$, and for the $(-2,3,7)$-pretzel knot one sees non-trivial invariants with residue characteristics $3$, $5$, $13$, $31$, $149$, $211$, $487$, $563$, and $34543$.
\end{exs}
%%%%%%%%%%%%%%%%%%%%

Understanding the local invariants of the algebras $A_\rho$ turns out to be linked with the problem of extending this association of a quaternion algebra $A_\rho$ to the character with an absolutely irreducible representation to an Azumaya algebra over normalizations of various subschemes of $X(\Gamma)$. To explain this we need some further notation.

In \S \ref{sec:BasicX} we recall a result from \cite{CullerShalen} showing that there is a canonical model $X(\Gamma)_\Q$ of $X(\Gamma)$ whose affine ring is the $\Q$-algebra generated by the trace functions associated with explicit words in the chosen generators for $\Gamma$. This choice of generators then fixes an embedding of $X(\Gamma)$ into an affine space $\mathbb{A}^n_{\mathbb{C}}$ where the indeterminates generating the affine ring of $\mathbb{A}^n_{\mathbb{C}}$ correspond to the trace functions associated with the given words in the fixed generating set.

Suppose $\mathfrak{C}$ is an integral complex curve that is a closed subscheme of $X(\Gamma)$, and suppose there is a point of $\mathfrak{C}(\mathbb{C})$ that corresponds to the character of an irreducible representation. By \cite[Ch.\ 3]{LangIntro} there is a unique minimal subfield $k\subset \mathbb{C}$ such that the ideal $I$ corresponding to the resulting embedding $\mathfrak{C} \subset \mathbb{A}^n_{\mathbb{C}}$ is generated by polynomials with coefficients in $k$. This field $k$ is called the field of definition of $\mathfrak{C}$ in \cite{LongReidDefine}.

Choosing generators for $I$ that have coefficients in $k$, we arrive at a geometrically integral curve $C$ over $k$ that is a closed subscheme of $X(\Gamma)_k = X(\Gamma)_\Q \otimes_\Q k$ such that $C \otimes_k \C$ is isomorphic to $\mathfrak{C}$ over $\mathbb{C}$. Since $C$ is geometrically integral $k$, is the field of constants of $C$, i.e., $k$ is algebraically closed in the function field $k(C)$ \cite[Rem.\ 9.5.7]{Poonen}. Note that this notion of $k$ being a field of definition depends on the realization of $\mathfrak{C}$ as a closed subscheme of $X(\Gamma)$, in contrast with other notions \cite{Debes}. It is shown in \cite[Prop.\ 3.1]{LongReidDefine} that $k$ and the isomorphism type of $C$ over $k$ do not depend on the choice of generators for $\Gamma$ used to define character function generators for the affine ring of $X(\Gam)_{\mathbb{Q}}$.

We now suppose that, in addition to the above, $k \subset \C$ is a number field. This will be the case, for example, if $\mathfrak{C}$ is an irreducible component of $X(\Gamma) = X(\Gamma)_\Q \otimes_\Q \C$. Let $C^\sharp$ be the normalization of $C$ and $\wt{C}$ be the unique smooth projective curve over $k$ containing $C^\sharp$ that has the same function field $k(C)$ as $C$ and $C^\sharp$. There is a field $F$ containing $k(C)$ and an (absolutely) irreducible representation $P_C: \Gam \to \SL_2(F)$ whose character defines the generic point of $C$. This representation is called the \emph{tautological representation} by Culler--Shalen \cite{CullerShalen}. The starting point for us is the following result, which we prove in \S \ref{sec:Harari} using work of Culler--Shalen, Harari \cite{Harari}, and Rumely \cite{Rumely2}, along with some classical results about Azumaya algebras.

%%%%%%%%%%%%%%%%%%%%
\begin{thm}\label{thm:FirstMain}
Suppose that $\Gam$ is a finitely generated group with $\SL_2(\C)$ character variety $X(\Gam) = X(\Gamma)_\Q \otimes_\Q \C$. Let $k\subset \C$ be a number field, and suppose that $C$ is a geometrically integral curve on $X(\Gamma)_k = X(\Gamma)_\Q \otimes_\Q k$ such that $\mathfrak{C} = C \otimes_k \C \subset X(\Gamma)$ has field of definition $k$. Let $C^\sharp$ be the normalization of $C$ and $\wt{C}$ be the smooth projective closure of $C^\sharp$. Suppose that $\mathfrak{C}$ contains the character of an irreducible representation of $\Gam$.
\begin{enumerate}

\item[1.] Taking the $k(C)$-span of $P_C(\Gamma)$ defines a $k(C)$-quaternion algebra $A_{k(C)} \subset \M_2(F)$ for some finite extension $F$ of $k(C)$.

\item[2.] For $z \in C^\sharp$, let $w = w(z)$ be the image of $z$ on $C$ and $\rho_w$ be a complex representation with character $w$. Fix an embedding of the residue field $k(w)$ into $\C$ extending our fixed embedding of $k$. Then $k_{\rho_w} \subseteq k(w) \subseteq k(z)$, and $k(w)$ is generated by $k$ and $k_{\rho_w}$. If $w$ is a smooth point of $C$ then $k(z) = k(w)$.

\item[3.] Suppose there is no Azumaya algebra $A_{\wt{C}}$ over $\wt{C}$ with generic fiber isomorphic to $A_{k(C)}$. Then there is no finite set $S$ of places of $\Q$ with the following property: The $k(w)$-quaternion algebra $A_\rho \otimes_{k_{\rho}} k(w)$ is unramified outside the places of $k(w)$ over $S$ for all but finitely many smooth points $w \in C(\overline{\Q})$ for which $\rho = \rho_{w}$ is absolutely irreducible.

\item[4.] Suppose there is an Azumaya algebra $A_{\wt{C}}$ over $\wt{C}$ with generic fiber isomorphic to $A_{k(C)}$. Then there is a finite set $S$ of places of $\Q$ such that the $k(z)$-quaternion algebra $A_\rho \otimes_{k_{\rho}} k(z)$ is unramified outside the places of $k(z)$ over $S$ for all points $z \in C^\sharp(\overline{\Q})$ such that $\rho = \rho_{w(z)}$ is absolutely irreducible.

\item[5.] If $A_{\wt{C}}$ exists, then its class in the Brauer group $\Br(\wt{C})$ is determined by the isomorphism class of $A_{k(C)}$ as a quaternion algebra over $k(C)$.

\end{enumerate}
\end{thm}
%%%%%%%%%%%%%%%%%%%%

In the remainder of the introduction we discuss the case of most interest to us, namely, where $K \subset S^3$ is a hyperbolic knot with complement $M = S^3 \ssm K$ and $\Gam = \pi_1(M)$. The work of Thurston \cite{Th} (see also \cite{CullerShalen}) shows that $X(\Gam)$ contains a distinguished curve $\mathfrak{C}_M$, a so-called \emph{canonical component}. This is an irreducible component of $X(\Gam)$ containing the character of a discrete and faithful representation associated with the complete hyperbolic structure on $M$. See \S \ref{sec:CharVars} for further discussion. As noted above, since $\mathfrak{C}_M$ is an irreducible component of $X(\Gamma)$, the field of constants of $\mathfrak{C}_M$ is a number field ${k}$, and there is a geometrically integral curve $C_M \subset X(\Gamma)_{k}$ whose base change to $\C$ is $\mathfrak{C}_M$. We will refer to $C_M$ as the canonical component of $X(\Gamma)_{k}$. Let $\wt{C}_M$ denote the normalization of a projective closure of $C_M$, so $\wt{C}_M$ is a smooth projective curve over ${k}$.

Among the results we prove in this paper, we will produce various sufficient conditions for $A_{\wt{C}_M}$ to exist. For example, in the case of hyperbolic knot complements, the existence of these Azumaya algebras is closely related to arithmetic properties of the Alexander polynomial of $K$. Recall that the Alexander polynomial is a generator of the Fitting ideal for the conjugation action of a meridian on the commutator subgroup of the knot group (e.g., see \cite{Neuwirth}). A primary theme of this paper is that the obstruction to $A_{\wt{C}_M}$ existing is related to points on $\wt{C}_M$ associated with characters of nonabelian reducible representations. In particular, it is a classical fact that nonabelian reducible representations of knot groups into $\SL_2(\C)$ are closely related to roots of $\Del_K(t)$; see \S\ref{ssec:Knots} for a precise discussion. We prove the following in \S \ref{sec:Proofs}.

%%%%%%%%%%%%%%%%%%%%
\begin{thm}\label{thm:MainKnot}
Let $K$ be a hyperbolic knot with $\Gam = \pi_1(S^3 \smallsetminus K)$, and suppose that its Alexander polynomial $\Del_K(t)$ satisfies:

\medskip

\noindent $(\star)$ for any root $z$ of $\Del_K(t)$ in an algebraic closure $\overline{\Q}$ of $\Q$ and $w$ a square root of $z$, we have an equality of fields $\Q(w) = \Q(w + w^{-1})$.

\medskip

\noindent
Then $A_{\wt{C}_M}$ exists for the canonical component $C_M \subset X(\Gam)_{k}$.
\end{thm}
%%%%%%%%%%%%%%%%%%%%

%%%%%%%%%%%%%%%%%%%%
\begin{rem}
While we only state Theorem \ref{thm:MainKnot} for the canonical component, our techniques can apply to give Azumaya algebras over other irreducible curve components of the $\SL_2(\C)$ character variety. Indeed, many of the facts about the canonical component used in the proof apply to other components, like the so-called \emph{norm curves} that appear in Boyer and Zhang's proof of the finite filling conjecture \cite{Boyer--Zhang}. As our primary applications of Theorem \ref{thm:MainKnot} are to points on the canonical component, we leave it to the motivated reader to make the necessary adjustments for producing Azumaya algebras over other components.
\end{rem}
%%%%%%%%%%%%%%%%%%%%

Theorem \ref{thm:MainKnot} obviously applies to any knot with trivial Alexander polynomial, and moreover to infinitely many other hyperbolic knots with non-trivial Alexander polynomial, including the figure-eight knot. Indeed, we are able to construct infinite families of both fibered and non-fibered hyperbolic knots for which condition $(\star)$ holds, and infinitely many for which it fails. See \S \ref{sec:Examples} for further discussion.

Furthermore, we prove a partial converse to Theorem \ref{thm:MainKnot}. For example, if $\mathfrak{C}_M$ is the unique component of $X(\Gam)$ containing the character of an irreducible representation, $\Delta_K(t)$ has no multiple roots, and $A_{\wt{C}_M}$ exists, then condition $(\star)$ of Theorem \ref{thm:MainKnot} holds for every root of $\Del_K(t)$. See \S \ref{sec:HeusenerPorti} for a precise discussion. The converse may also hold when $\Delta_K(t)$ has a multiple root, e.g., when the associated point on the character variety is a smooth point. In order to remove the multiple root condition and understand if and when the full converse holds, one must better understand the nature of singularities of $C$ arising from roots of high multiplicity.

A particularly interesting infinite family of points on $C_M$ are those that arise from performing hyperbolic Dehn surgery on $M$. Suppose that $N$ is a closed hyperbolic $3$-manifold obtained from Dehn surgery on $M$ and let $\chi_N$ be the character of the representation of $\Gam = \pi_1(M)$ obtained by composition of the Dehn surgery homomorphism and the faithful discrete representation of $\pi_1(N)$. If $k_N$ is the trace field of $N$, it is a well-known consequence of local rigidity that $k_N$ is a number field. As mentioned briefly above, there is a $k_N$-quaternion algebra $A_N$ associated with this point on $C_M$.

When Theorem \ref{thm:MainKnot} applies to show that we obtain an Azumaya algebra over $\wt{C}_M$, Theorem \ref{thm:FirstMain} then places considerable restrictions on the invariants of the algebras $A_N$. For the sake of simplicity, we state our results here for the case where $C_M$ has field of constants ${k} = \Q$. In fact, we do not know an example of a knot for which $C_M$ has field of constants not equal to $\Q$, and we will prove in Lemma \ref{lem:UniqueToQ} that ${k} = \Q$ when $X(\Gam)_\C$ has a unique component containing the character of an irreducible representation.

%%%%%%%%%%%%%%%%%%%%
\begin{thm}\label{thm:MainDehn}
Let $K$ be a hyperbolic knot in $S^3$ whose Alexander polynomial satisfies condition $(\star)$ of Theorem \ref{thm:MainKnot}, and suppose that the canonical component $C_M$ has field of constants $\Q$. Then there exists a finite set $S_K$ of rational primes such that for any hyperbolic Dehn surgery $N$ on $K$ with trace field $k_N$, the $k_N$-quaternion algebra $A_N$ can only ramify at real places of $k_N$ and finite places of $k_N$ lying over primes in $S_K$.
\end{thm}
%%%%%%%%%%%%%%%%%%%%

We now study effective upper bounds on the set $S_K$ in Theorem \ref{thm:MainDehn} using an integral version of Theorem \ref{thm:MainKnot}. Recall that by \cite{Lichtenbaum} (see also \cite{ChinburgMinimal} and \cite{Liu}), there is a regular projective model of $\wt{C}_M$ over the ring of integers $O_{k}$ of ${k}$. Such a model need not be unique, but there are relatively minimal models. If $\wt{C}_M$ has positive genus, then all relatively minimal models are isomorphic. One can take the base change of these models to Dedekind subrings of ${k}$ that contain $O_{k}$, for example rings of $S$-integers $O_{{k}, S}$, in order to arrive at regular projective models over these subrings.

%%%%%%%%%%%%%%%%%%%%
\begin{thm}\label{thm:MainKnotIntegral}
Suppose that $A_{\wt{C}_M}$ exists for the canonical component $C_M \subset X(\Gam)_{k}$, and let $S$ be a finite set of rational primes for which the following is true:
\medskip

\noindent $(\star_\ell)$ Let $\ell$ be a prime not in $S$. Suppose that $z$ is a root of $\Del_K(t)$ in an algebraic closure $\overline{\F}_\ell$ of $\F_\ell$. If $w$ is a square root of $z$, then we have an equality of fields $\F_\ell(w) = \F_\ell(w + w^{-1})$.

\medskip

\noindent For any such $S$, let $O_{{k},S}$ be the ring of $S$-integers of ${k}$ and let $\mathcal{C}_S$ be a regular projective integral model of $\wt{C}_M$ over $O_{{k},S}$. There is an extension $A_{\mathcal{C}_{S}}$ of $A_{k(C)}$ to an Azumaya algebra over $\mathcal{C}_{S}$, and the class of $A_{\mathcal{C}_{S}}$ in $\Br(\mathcal{C}_S)$ is determined by the isomorphism class of $A_{k(C)}$.
\end{thm}
%%%%%%%%%%%%%%%%%%%%

One can check (see Remark \ref{rem:cebot}) that if condition $(\star)$ of Theorem \ref{thm:MainKnot} holds, there will always be a finite set of primes $S$ as in Theorem \ref{thm:MainKnotIntegral}. In short, Theorem \ref{thm:MainKnotIntegral} says that \emph{ramification of the quaternion algebras associated with hyperbolic Dehn surgeries on $K$ are governed by the arithmetic of the Alexander polynomial of $K$}. A particular corollary of Theorems \ref{thm:MainDehn} and \ref{thm:MainKnotIntegral} is the following, noting that $(\star_\ell)$ holds for all primes $\ell$ when $\Del_K(t)=1$.

%%%%%%%%%%%%%%%%%%%%
\begin{cor}\label{cor:trivialalex}
Let $K\subset S^3$ be a hyperbolic knot with canonical component $C_M$ with field of constants ${k} = \Q$.
\begin{enumerate}

\item Suppose that $A_{\wt{C}_M}$ exists for $\wt{C}_M$. For any hyperbolic Dehn surgery $N$ on $K$ with trace field $k_N$, the $k_N$-quaternion algebra $A_N$ can only ramify at real places of $k_N$ and finite places lying over primes in the set $S$ provided by Theorem \ref{thm:MainKnotIntegral}.

\item Suppose that $\Del_K(t)=1$, and let $N$ be a closed hyperbolic $3$-manifold obtained by Dehn surgery on $K$. Then the quaternion algebra $A_N$ can only ramify at real places of the trace field $k_N$.

\end{enumerate}
\end{cor}
%%%%%%%%%%%%%%%%%%%%

As an example of Theorem \ref{thm:MainKnotIntegral}, we show in \S \ref{sec:Examples} that for the figure-eight knot complement, where $\wt{C}_M$ is known to be an elliptic curve over ${k} = \mathbb{Q}$ with good reduction outside $2$ and $5$, that the set $S$ in Theorem \ref{thm:MainKnotIntegral} can be taken to be $\{2\}$. This leads to the following consequence, which confirms the experimental observations described above.

%%%%%%%%%%%%%%%%%%%%
\begin{thm}\label{thm:Fig8}\ 
Let $K\subset S^3$ denote the figure-eight knot, $M=S^3 \smallsetminus K$, and
$C_M$ be the canonical component.
\begin{enumerate}
\item The quaternion algebra $A_{k(C)}$ extends to a quaternion Azumaya algebra $A_{\wt{C}_M}$.

\item The class of $A_{\wt{C}_M}$ in $\Br(\wt{C}_M)$ is the unique non-trivial class of order $2$ having trivial specialization at infinity that becomes trivial in $\Br(\wt{C}_M \otimes_{\mathbb{Q}} \mathbb{Q}(i))$ after tensoring over $\Q$ with $\Q(i)$.

\item Suppose that $N$ is a closed hyperbolic $3$-manifold obtained by Dehn surgery on the figure-eight knot. Then the quaternion algebra $A_N$ can only ramify at real or dyadic places of the trace field $k_N$.

\end{enumerate}
\end{thm}
%%%%%%%%%%%%%%%%%%%%

Another direction of interest is how our results are related to the existence of characters of $\SU(2)$ representations. The study of $\SU(2)$ representations of knot groups saw a great deal of activity motivated by an approach to Property P via the Casson invariant (see \cite{AM} and \cite{KM1}). In particular, if $\Sigma$ is an integral homology $3$-sphere that is obtained by Dehn surgery on a knot $K$ whose symmetrized Alexander polynomial $\Delta_K(t)$ satisfies $\Delta_K''(1)\neq 0$, then $\pi_1(\Sigma)$ admits a non-trivial homomorphism to $\SU(2)$, and so $K$ satisfies Property P. It was subsequently shown in \cite{KM2} that every non-trivial knot admits a curve of characters of irreducible $\SU(2)$-representations. Our work connects to this as follows (see \S \ref{sec:IntegralModel} for more detail).

%%%%%%%%%%%%%%%%%%%%
\begin{thm}\label{non-trivial}
Suppose $K$ is a hyperbolic knot with $\Delta_K(t)=1$. Let $C\subset X(\Gamma)_k$ be the canonical component and $\wt{C}$ be the smooth projective model of $C$, where $k$ is the {constant field} of the function field of $C$. Then $A_{\wt{C}}$ extends to an Azumaya algebra $\mathcal{A}$ over every regular integral model $\mathcal{C}$ of $\wt{C}$ over the full ring of integers $O_{k}$ of ${k}$. Let $C_{\SU(2)}$ be the subset of the real points $C(\mathbb{R})$ of $C$ corresponding to the characters of $\SU(2)$ representations and $C_{sing}$ be the (finite) singular locus of $C$.

\begin{enumerate}

\item[1.] The class $[\mathcal{A}]$ of $\mathcal{A}$ in $\Br(\mathcal{C})$ has an associated class $\beta([\mathcal{A}])$ in the relative Tate--Shafarevich group $\Sha({k},O_{{k}},\mathrm{Pic}^0(\wt{C}))$ defined by Stuhler in \cite[Def. 1, p. 149]{Stuhler} (see \S \ref{TSgroups}).

\item[2.] The following conditions are equivalent:
\begin{enumerate}
\item[i.] The class $\beta([\mathcal{A}])$ lies in the traditional Tate--Shafarevich group $\Sha({k},\mathrm{Pic}^0(\wt{C}))$ defined in Equation \eqref{eq:traditional} of \S \ref{TSgroups};
\item[ii.] $C_{\SU(2)}$ is contained in the finite set $C_{sing}$.
\end{enumerate}
If any (hence both) of these conditions fail, then $\wt{C}(\mathbb{R})$ is a finite non-empty disjoint union of real circles and $C_{\SU(2)} \ssm C_{sing}$ is a non-empty union of arcs and circles in $\wt{C}(\mathbb{R})$.

\item[3.] If either (i) or (ii) of part 2 above hold and there is a point of $\wt{C}$ over ${k}$, then $\beta([\mathcal{A}]) = 0$ in $\Sha({k},\mathrm{Pic}^0(\wt{C}))$ if and only if the class of $\mathcal{A}$ in $\Br(\mathcal{C})$ is trivial. In particular, this holds when $\wt{C}$ has no real points or when $\wt{C}(\R)$ consists only of real circles associated with $\SL_2(\R)$ representations.

\end{enumerate}
\end{thm}
%%%%%%%%%%%%%%%%%%%%

For an example of a reducible representation that is a singular point of the closure of the subscheme of irreducible representations on the character variety, see \cite[\S 6.2]{HP}. 

Theorem \ref{non-trivial} illustrates an approach to producing characters of irreducible $\SU(2)$-representations on the canonical component, albeit under the hypothesis $\Del_K(t)=1$. Among the large number of examples we have computed, we find the following reasonable (cf.\ \cite[\S5]{LongReidDefine}):

%%%%%%%%%%%%%%%%%%%%
\begin{conj}\label{conj:findsu2}
Let $K$ be a hyperbolic knot in $S^3$. Then there is a real curve of $\SU(2)$ characters contained in the canonical component of the $\SL_2(\C)$ character variety.
\end{conj}
%%%%%%%%%%%%%%%%%%%%

It would be very interesting if an approach via the techniques of this paper could be used to give an alternate proof that knots groups have irreducible $\SU(2)$ representations, as opposed to the gauge-theoretical methods of \cite{KM1, KM2}. In contrast with the discussion here, it is worth pointing out that there are constructions of $1$-cusped hyperbolic 3-manifolds for which the canonical component does not contain any real characters (see for example \cite{LongReidDefine}).

Furthermore, Theorem \ref{non-trivial} also ties Conjecture \ref{conj:findsu2} to the question as to which curves can possibly be the canonical component of the $\SL_2(\C)$ character variety of a hyperbolic knot. An immediate consequence of Theorem \ref{non-trivial} is the following special case of Conjecture \ref{conj:findsu2}.

%%%%%%%%%%%%%%%%%%%%
\begin{cor}\label{cor:findsu2}
Let $K$ be a hyperbolic knot with trivial Alexander polynomial. Let $C\subset X(\Gam)_k$ be the canonical component, where $k$ is the field of constants of the function field $k(C)$. Suppose that $\wt{C}$ is a smooth projective model of $C$, $\wt{C}$ has a point over ${k}$, and that $\mathrm{Pic}^0(\wt{C})$ is the Jacobian of $\wt{C}$. If the Tate--Shafarevich group $\Sha({k},\mathrm{Pic}^0(\wt{C}))$ of $\mathrm{Pic}^0(\wt{C})$ is trivial, then either $A_{\wt{C}}$ is isomorphic to $\M_2(k(C))$ or $C$ contains infinitely many characters of non-abelian $\SU(2)$ representations.
\end{cor}
%%%%%%%%%%%%%%%%%%%%

Our work also seems to suggest a connection between hyperbolic knot complements whose canonical components satisfy the conclusion of Theorem \ref{thm:MainKnot}, $L$-space knots, and knots whose complements have bi-orderable fundamental group. We discuss this in more detail in \S \ref{sec:Examples}, but in rough terms we do not know of an example of a hyperbolic $L$-space knot that satisfies the conclusion of Theorem \ref{thm:MainKnot}. In a similar direction, in \S \ref{sec:Examples} we discuss some connections between our work, (non-)orderability of hyperbolic knot groups, and fundamental groups of Dehn surgeries on knots. See \cite{CullerDunfield} for further work related to this connection.

\medskip

\noindent{\bf Acknowledgements:} T.\ C.\ was partially supported by NSF FRG Grant DMS-1360767, NSF FRG Grant DMS-1265290, NSF SaTC grant CNS-1513671, and Simons Foundation Grant 338379. T.\ C.\ would also like to thank L'Institut des Hautes \'{E}tudes Scientifiques for support during the Fall of 2015. 
A.\ W.\ R.\ was partially supported by NSF Grant DMS-1105002 and NSF FRG Grant DMS-1463797. He also acknowledges the support of NSF Grant DMS-1440140 whilst in residence at M.S.R.I. for part of Fall 2017, as well as the Wolfensohn Fund whilst he was in residence at The Institute for Advanced Study in Spring 2017. He would like to thank both these institutions for their support and hospitality.
M.\ S.\ was partially supported by NSF Grant DMS-1361000, Grant Number 523197 from the Simons Foundation/SFARI, and NSF Grant DMS-1906088. The authors acknowledge support from U.S.\ National Science Foundation grants DMS 1107452, 1107263, 1107367 "RNMS: GEometric structures And Representation varieties" (the GEAR Network).

The authors also thank the following people: Ian Agol for conversations about $\SU(2)$ representations; Nathan Dunfield for pointing out \cite{Konvalina--Matache}; Darren Long for a number of conversations related to this paper; Raman Parimala and Suresh Venapally for conversations about Brauer groups of curves. Finally, we thank the referee for very helpful comments.

%%%%%%%%%%%%%%%%%%%%
\section{Representation and character varieties}\label{sec:CharVars}
%%%%%%%%%%%%%%%%%%%%

In this section, we recall some basic facts about $\SL_2(\C)$ representation and character varieties.

%%%%%%%%%%%%%%%%%%%%
\subsection{}\label{sec:BasicX}
%%%%%%%%%%%%%%%%%%%%

Let $\Gamma$ be a finitely generated group, and consider the $\SL_2(\C)$ \emph{representation variety}
\[
R(\Gam) = \mathrm{Hom}(\Gam, \SL_2(\C)).
\]
The embedding $\SL_2(\C) \subset \M_2(\C) \cong \C^4$ with the obvious coordinates gives $R(\Gam)$ the structure of an affine algebraic subset of $\C^{4 n}$. These coordinates show that the embedding $R(\Gamma) \subset \C^{4n}$ is the base change of an embedding $R(\Gamma)_\Q \subset \mathbb{A}^{4n}_\Q$, where $R(\Gamma)$ and $R(\Gamma)_\Q$ are up to canonical isomorphism independent of the choice of a generating set for $\Gamma$. For all fields $L$ of characteristic $0$, we will denote $R(\Gamma)_\Q \otimes_{\Q} L$ by $R(\Gamma)_L$.

Two elements $\rho_1, \rho_2 \in R(\Gam)$ are called \emph{equivalent} if there is some $g \in \GL_2(\C)$ such that $\rho_2 = g \rho_1 g^{-1}$. A representation $\rho \in R(\Gam)$ is called \emph{reducible} if $\rho(\Gam)$ is conjugate into the subgroup of upper-triangular matrices and is called \emph{irreducible} if it is not reducible.

Let $\chi_\rho:\Gam\rightarrow \C$ denote the character of a representation $\rho$. For each $\gam \in \Gam$, consider the regular function $\tau_\gam : R(\Gam) \to \C$ defined by evaluating the character of $\rho$ at $\gam$:
\[
\tau_\gam(\rho) = \chi_\rho(\gam) = \tr(\rho(\gam))
\]
Since the trace is a conjugacy invariant, $\tau_\gam$ is constant on equivalence classes of representations. The subring $T$ of the ring of regular functions on $R(\Gam)$ generated by $\{\tau_\gam\}_{\gam \in \Gam}$ is finitely generated \cite[Prop.\ 1.4.1]{CullerShalen}. Therefore, we can fix $\gam_1, \dots, \gam_r \in \Gam$ such that $\{\tau_{\gam_i}\}_{i = 1}^r$ generates $T$.

Define $t : R(\Gam) \to \C^r$ by
\[
t(\rho) = \big( \tau_{\gam_1}(\rho), \dots, \tau_{\gam_r}(\rho) \big) \in \C^r.
\]
Note that if $\rho_1, \rho_2$ are equivalent representations, then $t(\rho_1) = t(\rho_2)$. The $\SL_2(\C)$ \emph{character variety} of $\Gam$ is
\[
X(\Gam) = t(R(\Gam)) \subseteq \C^r,
\]
and every irreducible component of $X(\Gam)$ containing the character of an irreducible representation is a closed affine algebraic variety \cite[Prop.\ 1.4.4]{CullerShalen}. For $\gam \in \Gam$, we also define the rational function
\[
I_\gam(\chi_\rho) = \chi_\rho(\gam),
\]
on $X(\Gam)$ induced by the class function $\tau_\gam$ on $R(\Gam)$ defined above.

To be more precise, \cite{CullerShalen} shows that $X(\Gam)$ has affine coordinate ring $\C[x_1,\ldots,x_r]/J$, where $J$ is the ideal of all polynomials that vanish on $X(\Gam)$ under the identification $x_i = \tau_{\gamma_i}$. Changing the generating set gives an isomorphic affine set, so $X(\Gam)$ is well-defined up to isomorphism. It is also shown in \cite{CullerShalen} that $X(\Gam)$ is defined over $\Q$ in the sense that $J$ is generated by polynomials in the variables $x_i$ with coefficients in $\Q$. We will use $X(\Gamma)_\Q$ to denote the affine scheme over $\Q$ whose ring is the image of $\Q[x_1,\ldots,x_r]$ in $\C[x_1,\ldots,x_r]/J$. The irreducible components of $X(\Gamma)_\Q$ then have constant field a number field.

If $F$ is an arbitrary field of characteristic zero, then a representation $\rho:\Gamma\rightarrow \SL_2(F)$ is {\em absolutely irreducible} if it remains irreducible over an algebraic closure $\overline{F}$ of $F$. A representation $\rho:\Gamma\rightarrow \SL_2(F)$ with non-abelian image is absolutely irreducible over an algebraic closure of $F$ if and only if $\rho$ is irreducible \cite[Lem.\ 1.2.1]{CullerShalen}.

Two irreducible representations of $\Gam$ are equivalent if and only if they have the same character \cite[Prop.\ 1.5.2]{CullerShalen}. In particular, if $x \in X(\Gam)$ is a point such that $x = t(\rho)$ for some irreducible representation $\rho \in R(\Gam)$, then $t^{-1}(x)$ is exactly the equivalence class of $\rho$. Furthermore, the reducible representations in $R(\Gam)$ are of the form $t^{-1}(V)$ for some closed algebraic subset $V$ of $X(\Gam)$ \cite[Prop.\ 1.4.2]{CullerShalen}.

We now record the following from \cite[Lem.\ 1.2.1]{CullerShalen}.

%%%%%%%%%%%%%%%%%%%%
\begin{lem}\label{trace2}
In the above notation, if $\chi_\rho$ is the character of a reducible representation and $c\in [\Gam,\Gam]$, then $I_c(\chi_\rho)=2$.
\end{lem}
%%%%%%%%%%%%%%%%%%%%

%%%%%%%%%%%%%%%%%%%%
\begin{proof}
A reducible representation of $\Gam$ into $\SL_2(\C)$ can be conjugated to have image contained in the group of upper-triangular matrices. Since the commutator of two upper-triangular matrices has $1$s on the diagonal, the lemma follows.
\end{proof}
%%%%%%%%%%%%%%%%%%%%

%%%%%%%%%%%%%%%%%%%%
\begin{rem}\label{rem:reducible}
The converse of Lemma \ref{trace2} also holds. See \cite[Lem.\ 1.2.1]{CullerShalen}. Indeed this holds for representations to $\SL_2(F)$ for an arbitrary algebraically closed field $F$.
\end{rem}
%%%%%%%%%%%%%%%%%%%%

\medskip

\noindent\textbf{Notation:}~In the case when $\Gamma$ is the fundamental group of the complement $S^3\ssm K$ of a knot $K$ in the $3$-sphere $S^3$, we denote the representation variety by $R(K)$ and the character variety by $X(K)$.

%%%%%%%%%%%%%%%%%%%%
\subsection{}\label{ssec:TreeAction}
%%%%%%%%%%%%%%%%%%%%

Throughout what follows, by an affine or projective curve we shall always mean an irreducible affine or projective curve.

Suppose that $C$ is an affine curve with field of constants the number field $k$. For us, $C$ will generally be a closed subscheme of $R(\Gam)_k$ or $X(\Gam)_k$. Let $C^\#$ denote the normalization of the reduction $C^{red}$ of $C$. Thus if $C = \Spec(A)$ we have $C^{red} = \Spec(A^{red})$ and $C^\# = \Spec(A^\#)$, where $A^{red}$ is the quotient of $A$ by its nilradical and $A^\#$ is the normalization of $A^{red}$ in the function field $k(C)$ of $C$. The natural morphsim $C^\# \to C$ is finite, and $C^\#$ is connected since $C$ is irreducible.

Denote the smooth projective completion of $C^\#$ by $\wt{C}$, so $\wt{C}$ is a smooth projective curve birationally equivalent to $C$ that contains $C^\#$ as an open dense subset. The \emph{ideal points} of $\wt{C}$ are $\wt{C} \smallsetminus C^\#$, which are the points at which the birational map $\wt{C}\rightarrow C$ is not defined. We denote the set of ideal points by $\mathcal{I}(\wt{C})$. Notice that a non-zero regular function $C \otimes_k \C \to \C$ induces a map $\wt{C} \otimes_k \C \to \mathbb{P}^1_{\C}$ whose poles are at points in $\mathcal{I}(\wt{C})$.

The following is a generalization of a fact implicit in Culler--Shalen \cite{CullerShalen} in the case where $C$ is an irreducible curve components of $X(\Gam)$.

%%%%%%%%%%%%%%%%%%%%
\begin{lem}\label{lem:BuildTautological}
The morphism $R(\Gam)_k \to X(\Gam)_k$ is surjective. Suppose that $\eta_C$ is the generic point of an irreducible curve $C \subset X(\Gam)_k$. Then there is an irreducible curve $D \subset R(\Gam)_k$ such that $t(\eta_D) = \eta_C$ and $t(D) \subset C$, where $\eta_D$ denotes the generic point of $D$. The function field $k(D)$ of $D$ is a finite extension of the function field $k(C)$ of $C$. Further, there exists a representation $P_C : \Gam \to \SL_2(k(D))$ such that
\[
\chi_{P_C}(\gam)(\rho) = \chi_\rho(\gam)
\]
for any representation $\rho \in D$ and $\gam \in \Gam$. In other words, evaluating the function $\chi_{P_C}(\gam) \in k(D)$ at the point $\rho$ gives the value of the character $\chi_\rho$ at $\gam$.
\end{lem}
%%%%%%%%%%%%%%%%%%%%

%%%%%%%%%%%%%%%%%%%%
\begin{proof}
By Exercises II.3.18 and II.3.19 in \cite{Hartshorne}, the image of the morphism $R(\Gam)_k \to X(\Gam)_k$ is constructible, so it is the finite disjoint union of locally closed subsets of $X(\Gam)_k$. Therefore $t : R(\Gam)_k \to X(\Gam)_k$ is surjective as a map of affine schemes, since the base change
\[
R(\Gam) = R(\Gam)_k \otimes_k \C \to X(\Gam) = X(\Gam)_k \otimes_k \C
\]
is surjective on closed points by \cite[p.\ 117]{CullerShalen}. Therefore the fiber $t^{-1}(\eta_C)$ of $t$ over $\eta_C$ is a non-empty scheme over the residue field $k(\eta_C) = k(C)$. It therefore has a closed point $\eta_D$ as a scheme over the field $k(\eta_C)$. Now, $k(\eta_D)$ is a finite extension of $k(\eta_C)$ and the Zariksi closure $D$ of $\eta_D$ in $R(\Gam)_k$ is a curve in $R_k$ such that $t(D) \subset C$.

We let the representation $P_C$ be the one produced by $\eta_D$. Specifically, we arrive at a so-called \emph{tautological representation} $P_C: \Gam \to \SL_2(k(D))$, which we denote by
\[
P_C(\gam) = \begin{pmatrix} f_\gam^{1, 1} & f_\gam^{1, 2} \\ & \\ f_\gam^{2, 1} & f_\gam^{2, 2} \end{pmatrix},
\]
where $f_\gam^{i,j}\in k(D)$ is the function such that $f_\gam^{i,j}(\rho)$ is the $(i,j)$-entry of $\rho(\gam)$. The character of $P_C$ visibly has the property that $\chi_{P_C}(\gam)(\rho) = \chi_\rho(\gam)$ for all $\rho \in D$ and $\gam \in \Gam$. This proves the lemma.
\end{proof}
%%%%%%%%%%%%%%%%%%%%

Given the representation $P_C$ described above, we record the following basic lemma that follows from the reasoning used in the proof of Lemma \ref{lem:BuildTautological} (also see \cite[Lem.\ 1.3.1]{CullerShalen}).

%%%%%%%%%%%%%%%%%%%%
\begin{lem}\label{taut_irred}
In the above notation, if $C$ contains the character of an irreducible representation, then the representation $P_C$ from Lemma \ref{lem:BuildTautological} is (absolutely) irreducible.
\end{lem}
%%%%%%%%%%%%%%%%%%%%

\noindent{\bf Notation:}~We recall the following notation. Let $C$ be a possibly singular projective curve with smooth projective model $\wt{C}$ and $P\in \widetilde{C}$. Then for $\alpha = f/g \in k(C)$, with $f,g\in k[C]$ regular functions we set
\[
\ord_P(\alpha) = \ord_P(f) - \ord_P(g),
\]
where $\ord_P(f)$ (resp.\ $\ord_P(g)$) is the order of vanishing of $f$ (resp.\ $g$) at $P$. Then $\ord_P$ can be used to define a valuation of $k(C)$ at $P$ for which the {\em local ring} at $P$, will be denoted by $\mathcal{O}_P$ and consists of those $\alpha$ with $\ord_P(\alpha) \geq 0$, and its unique maximal ideal, denoted by $\mathfrak{m}_P$, consists of those $\alpha$ with $\ord_P(\alpha) > 0$. The residue field of the point $P$ is given by $k(P) = \mathcal{O}_P / \mathfrak{m}_P$.

\medskip

We introduce the following additional notation. Note that any function on $C$ in the ring of functions generated by the character functions $I_\gamma$ extends to a rational function $\wt{C}\rightarrow \mathbb{P}^1$. For $\gam \in \Gam$, let $\widetilde{I}_\gam : \widetilde{C} \to \mathbb{P}^1$ be the rational function induced by
\[
I_\gam(\chi_\rho) = \chi_\rho(\gam),
\]
where $\chi_\rho \in C$. Recall that $I_\gam$ is the function on $C$ induced by the class function $\tau_\gam$ on $R(\Gam)$ defined above. We will also frequently consider the related function
\[
f_\gamma(\chi_\rho) = \tr(\rho(\gamma))^2-4 = I_\gamma(\chi_\rho)^2-4,
\]
which vanishes precisely when $\rho(\gamma)$ is either unipotent or central.

With $C$ and $D$ as in the previous discussion, we obtain a finite morphism $\widetilde{t}: \widetilde{D} \to \widetilde{C}$. The ideal points on $\widetilde{D}$ are the inverse images of the ideal points of $\widetilde{C}$ under this map.

For a point $p\in \widetilde{D}$ with $\widetilde{t}(p)=q$, we have an associated local ring $\mathcal{O}_p$. If $\gam \in \Gam$, then as shown in \cite[Thm.\ 2.2.1]{CullerShalen}, the following conditions are equivalent:
\begin{enumerate}

\item[i.] every $P_C(\gam)\subset \SL_2(F)$ is $\GL_2(F)$-conjugate to an element of $\SL_2(\mathcal{O}_p)$;

\item[ii.] $\widetilde{I}_\gam$ does not have a pole at $q$.

\end{enumerate}
Since the functions $\widetilde{I}_\gam$ over all $\gam \in \Gam$ generate the ring of regular functions on $X(\Gam)$, for any $x\in \mathcal{I}(\widetilde{C})$ we can find a non-trivial $\gam \in \Gam$ such that $\widetilde{I}_\gam$ has a pole at $x$.

We now prove a slight strengthening of part (2) of Theorem \ref{thm:FirstMain}.

%%%%%%%%%%%%%%%%%%%%
\begin{lem}\label{lem:ResidueFieldTraceField}
Let $X(\Gam)_\Q$ be the model of $X(\Gam)$ over $\Q$ discussed in the paragraph prior to Theorem \ref{thm:FirstMain} of the introduction. Suppose that $k$ is a subfield of $\C$ and $C$ is an irreducible curve over $k$ that is a closed subscheme of the base change $X(\Gam)_\Q \otimes_\Q k$ of $X(\Gam)_\Q$ to $k$. For any $z \in C^\sharp = \wt{C} \ssm \mathcal{I}(\wt{C})$, let $w = w(z)$ be the image of $z$ on $C$, $\rho_w$ be a complex representation associated with $w$ and an embedding over $k$ of the residue field $k(w)$ into $\C$, and let $k_{\rho_w}$ be the trace field of $\rho_w$. Then
\[
k_{\rho_w} \subseteq k(w) \subseteq k(z),
\]
$k(z) = k(w)$ if $w$ is a smooth point of $C$, and $k(w)$ is generated by $k$ and $k_{\rho_w}$.
\end{lem}
%%%%%%%%%%%%%%%%%%%%

%%%%%%%%%%%%%%%%%%%%
\begin{proof}
By \cite{CullerShalen}, the affine ring of $X(\Gam)_\Q$ is generated as a $\Q$-algebra by trace functions associated to elements of $\Gamma$. Therefore when we view $w = w(z)$ as a closed point of $X(\Gam)_\Q \otimes_\Q k$, the residue field $k(w)$ is generated by $k$ and polynomial expressions in the values of trace functions at $w$. This shows that $k(w)$ is generated by $k$ and $k_{\rho_w}$, and the rest of the assertions in the statement of the lemma are clear.
\end{proof}
%%%%%%%%%%%%%%%%%%%%

%%%%%%%%%%%%%%%%%%%%
\subsection{One-cusped hyperbolic 3-manifolds}
%%%%%%%%%%%%%%%%%%%%

We now specialize some of the above discussion to the case of most interest to us, namely hyperbolic 3-manifolds. Throughout this paper, a hyperbolic $3$-manifold will always mean a connected, oriented, and complete manifold $M$ of the form $\Hthree / \Gam$, where $\Gam \cong \pi_1(M)$ is a torsion-free discrete subgroup of $\Isom^+(\Hthree) \cong \PSL_2(\C)$.

If $M$ is finite volume but not compact, then $M$ is the interior of a compact, irreducible, 3-manifold whose boundary is a finite union of incompressible tori. In this case there is a discrete and faithful representation $\rho_0 : \Gam \to \PSL_2(\C)$ coming from the holonomy of the complete structure on $M$, and local rigidity implies that any other discrete and faithful representation of $\Gam$ into $\PSL_2(\C)$ is equivalent to $\rho_0$, so we can speak of `the' discrete and faithful representation of $\Gam$. Thurston showed that $\rho_0$ lifts to a representation $\widehat{\rho}_0 : \Gam \to \SL_2(\C)$ \cite[Prop.\ 3.1.1]{CullerShalen}. In general there will be several lifts of $\rho_0$, even up to equivalence, however for us it will not matter which lift we consider, and similarly for a character $\chi_{\rho_0}$.

Define a \emph{canonical component} $X_0(\Gam) \subset X(\Gam)$ to be an irreducible component of $X(\Gam)$ containing some $\chi_{\rho_0}$. In particular, $X_0(\Gam)$ is an affine algebraic variety with field of constants a number field. When $M$ is a non-compact hyperbolic $3$-manifold of finite volume with $d$ ends homeomorphic to $T^2 \times [0, \infty)$ (where $T^2$ is the $2$-torus), Thurston showed that $X_0(\Gam)$ has complex dimension exactly $d$ \cite[Prop.\ 3.2.1]{CullerShalen}. We summarize the above discussion in the following result.

%%%%%%%%%%%%%%%%%%%%
\begin{thm}\label{thm:CanonCpt}
Let $M$ be a non-compact hyperbolic $3$-manifold of finite volume with $d$ ends and set $\Gam = \pi_1(M)$. Then any canonical component $X_0(\Gam)$ is a $d$-dimensional affine algebraic variety with field of constants a number field. In particular, when $M$ is a one-cusped hyperbolic $3$-manifold, the canonical component $X_0(\Gam)$ is an affine curve.
\end{thm}
%%%%%%%%%%%%%%%%%%%%

We have the following important lemma regarding the tautological representation
in the case of a canonical component.

%%%%%%%%%%%%%%%%%%%%
\begin{lem}\label{taut_faithful}
Let $M$ be a 1-cusped hyperbolic 3-manifold and $C$ the canonical component. Then the tautological representation $P_C$ is faithful.
\end{lem}
%%%%%%%%%%%%%%%%%%%%

%%%%%%%%%%%%%%%%%%%%
\begin{proof}
Suppose that $P_C$ is not faithful. Then there exists a non-trivial $\gamma\in \pi_1(M)$ such that $P_C(\gamma)=I$. In particular, this means that $\chi_\rho(\gamma)=2$ for all $\chi_\rho\in C$. Since $C$ is a canonical curve, the only non-trivial elements of $\pi_1(M)$ with trace $2$ under a faithful discrete representation are peripheral elements, hence $\gamma$ is peripheral. However, \cite[Prop.\ 1.1.1]{CGLS} implies that the function $I_\gamma$ is non-constant on $C$ when $\gamma$ is a non-trivial peripheral element, and this is a contradiction.
\end{proof}
%%%%%%%%%%%%%%%%%%%%

%%%%%%%%%%%%%%%%%%%%
\subsection{}\label{sec:InvQuat}
%%%%%%%%%%%%%%%%%%%%

Given a non-elementary subgroup $H$ of $\SL_2(\C)$ we can associate a field and quaternion algebra as follows (see \cite[Ch.\ 3]{MaclachlanReid}). The \emph{trace field} of $H$ is the field $k_H=\Q(\tr(\gamma) : \gamma\in H)$ and the quaternion algebra is
\[
A_H = \left\{ \sum_{i = 1}^n \alpha_i \gam_i\ :\ \alpha_i \in k_H, \gam_i \in H \right\},
\]
i.e., the $k_H$-span of $H$ in $\M_2(\C)$.

If $H$ is a Kleinian group of finite co-volume then $k_H$ is a number field. Let $H^{(2)}$ denote the (finite index) subgroup of $H$ generated by the squares of all elements in $H$; this is the kernel of the homomorphism from $H$ onto its maximal elementary $2$-abelian quotient. The {\em invariant trace field and quaternion algebra} associated with a finitely generated non-elementary subgroup are $kH = k_{H^{(2)}}$ and $AH = A_{H^{(2)}}$. These are invariants of the commensurability class of $H$ in $\PSL_2(\C)$, and $k\Gam$ is also equal to $\mathrm{tr}(\mathrm{Ad}(\Gam))$. When $H_1(H,\mathbb{F}_2)=\{0\}$ or when $H$ is the fundamental group of a knot complement in an integral homology sphere, the invariant trace field and quaternion algebra coincide with the trace field and the algebra $A_H$. See \cite[\S 4.2]{MaclachlanReid}.

A Hilbert symbol (see \cite[p.\ 78]{MaclachlanReid} for the definition) for $A_H$ is readily described using a pair of non-commuting elements as follows. Suppose that $g$ and $h$ are hyperbolic elements of $H$ with $[g,h]\neq 1$. Then, following \cite[\S 3.6]{MaclachlanReid}, a Hilbert symbol for $A_H$ is given by
\[
\left( \frac{\tr^2(g)-4,\tr([g,h])-2}{k_H} \right).
\]

%%%%%%%%%%%%%%%%%%%%
\subsection{}\label{sec:CanonQuat}
%%%%%%%%%%%%%%%%%%%%

We now define the quaternion algebra $A_{k(C)}$ over the function field $k(C)$ of $C$ that will be the central object of study in this paper.

We begin with some general comments in the setting finitely generated groups. Let $\Gamma$ be a finitely generated group and $C$ a geometrically irreducible curve over a number field $k$ that is a closed subscheme of $X(\Gamma)_k$, and suppose that $C$ contains the character of an irreducible representation. As in Lemma \ref{lem:BuildTautological}, fix an irreducible curve $D\subset R(\Gam)_k$ such that $t(D)=C$ and the function field $F = k(D)$ of $D$ is a finite extension of the function field $k(C)$ of $C$. As above, we have the tautological representation $P_C:\Gam\rightarrow \SL_2(F)$.

Since $C$ contains the character of an irreducible representation, we know from Lemma \ref{taut_irred} that $P_C$ is absolutely irreducible. We can then define $A_{k(C)}$ to be the $k(C)$-subalgebra of $\M_2(F)$ generated by the elements of $P_C(\Gam)$. That is,
\[
A_{k(C)} = \left\{ \sum_{i = 1}^n \alpha_i P_C(\gam_i)\ :\ \alpha_i \in k(C), \gam_i \in \Gam \right\}.
\]
Exactly as in the proof of Lemma \ref{lem:BuildTautological}, $A_{k(C)}$ has the structure of a quaternion algebra over $k(C)$. We refer to $A_{k(C)}$ as the {\em canonical quaternion algebra}. It will be helpful to record part of the proof of this, namely that $A_{k(C)}$ is $4$-dimensional over $k(C)$, by identifying certain elements of $\Gam$ whose images under $P_C$ provide a $k(C)$-basis.

%%%%%%%%%%%%%%%%%%%%
\begin{lem}\label{non-commuting}
In the notation above, there exists a pair of elements $g,h\in \Gam$ so that the regular functions $I_g^2 - 4$ and $I_{[g,h]}-2$ are not identically zero on $\wt{C}$. More specifically, given any $g\in \Gam$ so that $I_g$ is not constant with value $\pm 2$ on $\wt{C}$, there is an element $h\in \Gam$ so that $I_{[g,h]}-2$ is not identically zero on $\wt{C}$.
\end{lem}
%%%%%%%%%%%%%%%%%%%%

%%%%%%%%%%%%%%%%%%%%
\begin{proof}
Since $P_C$ is irreducible it has non-abelian image, so there exists $g\in\Gam$ so that $P_C(g)\neq \pm I$. As argued in \cite[Lem.\ 1.5.1]{CullerShalen}, we can find $h\in\Gam$ so that $P_C$ restricted to the subgroup $H$ generated by $g$ and $h$ is irreducible and $\chi_\rho(h)\neq \pm 2$ for all $\chi_\rho\in C$.

It follows from irreducibility that $I_{[g,h]}-2\neq 0$. Indeed, by assumption we can conjugate $P_C(g)$ over the algebraic closure of $k(C)$ to be the diagonal matrix
\[
P_C(g) = \begin{pmatrix} u & 0 \\ & \\ 0 & 1/u \end{pmatrix}
\]
for some function $u\neq \pm 1$. Then,
\[
P_C(h) = \begin{pmatrix} a & b \\ & \\ c & d \end{pmatrix}
\]
for functions $a,b,c,d$ in the algebraic closure of $k(C)$. Computing $\tr([P_C(g),P_C(h)])$ and setting this equal to $2$ we obtain the equation
\[
bc(2-(u+1/u))=0.
\]
It follows that $bc=0$. However this cannot be the case, as it would then follow that the restriction of $P_C$ to $H$ is either upper- or lower-triangular, i.e., reducible on $H$, which is a contradiction. This proves the first part of the lemma.

The second part follows the same line of argument after noticing that $I_h$ not being constant with value $\pm 2$ implies that $P_C(h)\neq \pm I$.
\end{proof}
%%%%%%%%%%%%%%%%%%%%

When the tautological representation $P_C$ is (absolutely) irreducible, using Lemma \ref{non-commuting} and following \cite[\S 3.6]{MaclachlanReid} we see that there exist elements $\{g,h\} \in \Gam$ so that $\{1,P_C(g),P_C(h),P_C(gh)\}$ is a basis for $A_{k(C)}$ over $k(C)$. With this one can describe a Hilbert symbol (cf.\ \S \ref{sec:InvQuat}).

%%%%%%%%%%%%%%%%%%%%
\begin{cor}\label{hilberttaut}
Let $\Gamma$ be a finitely generated group and $C$ be a geometrically integral curve over $k$ that is a closed subscheme of $X(\Gamma)_k$. Assume that $C$ contains the character of an irreducible representation, and let $g, h \in \Gam$ be two elements such that there exists a representation $\rho \in R(\Gam)$ with character $\chi_\rho \in C$ for which the restriction of $\rho$ to $\langle g, h \rangle$ is irreducible. Then the canonical quaternion algebra $A_{k(C)}$ is described by the Hilbert symbol
\[
\left( \frac{I_g^2 - 4\, ,\, I_{[g,h]}-2}{k(C)} \right).
\]
\end{cor}
%%%%%%%%%%%%%%%%%%%%

In \S \ref{sec:Azumaya}, we will describe how the quaternion algebra $A_{k(C)}$ can be described instead as an Azumaya algebra over $k(C)$. This will provide the correct context for the above discussion to be applied to prove our main results.

%%%%%%%%%%%%%%%%%%%%
\subsection{Knot complements}\label{ssec:Knots}
%%%%%%%%%%%%%%%%%%%%

We now specialize some of the previous discussion to hyperbolic knot
complements. We fix the following notation for the remainder of this paper. If $K \subset S^3$ is a non-trivial knot, $E(K)$ will denote the exterior of $K$. We fix a standard pair of preferred generators for $\pi_1(\partial E(K))$, namely $\langle \mu,\lambda \rangle$ where $\mu$ is a meridian of the knot $K$ and $\lambda$ a longitude (chosen to be null-homologous in $E(K)$). Elements of $\Gam$ conjugate into $\langle \mu,\lambda \rangle$ are called \emph{peripheral elements}. The Alexander polynomial of $K$ will be denoted by $\Delta_K(t)$. Recall from the discussion in the introduction that the Alexander polynomial is a generator of the Fitting ideal for the conjugation action of a meridian on the commutator subgroup of the knot group. See \cite{Neuwirth} for further background on the Alexander polynomial sufficient to understand the results in this paper.

For much of the rest of this paper we will be interested in character varieties of hyperbolic knot complements, and in particular their canonical components. We point out that it is known from \cite{KM2} that if $K$ is \emph{any} non-trivial knot, then the character variety contains a curve of characters of irreducible representations. However, our focus is on hyperbolic knots.

Thus, let $K\subset S^3$ be a hyperbolic knot. As remarked upon in \S \ref{sec:BasicX}, characters of reducible representations of $\Gam=\pi_1(S^3\ssm K)$ form an algebraic subset of $X(K)$. Denote this subset by $X_R(K)$. We will use the following fact.

%%%%%%%%%%%%%%%%%%%%
\begin{lem}\label{finitelymanyreducible}
Let $\mathfrak{C} \subset X(K)$ be the canonical component over the complex numbers. In the above notation, $X_R(K)\cap \mathfrak{C}$ can consist of only finitely many points. 
\end{lem}
%%%%%%%%%%%%%%%%%%%%

%%%%%%%%%%%%%%%%%%%%
\begin{proof}
As above, let $\lambda$ denote a longitude of $K$. Since $\lambda \in [\Gam,\Gam]$, Lemma \ref{trace2} shows that $I_\lambda(\chi_\rho)=2$ for any reducible representation $\rho$. Since $\mathfrak{C}$ is a curve, if $\mathfrak{C}$ contained the characters of infinitely many reducible representations it would follow that $I_\lambda(\chi_\rho)=2$ for all $\chi_\rho\in \mathfrak{C}$. However, this is impossible, since the functions $I_\alpha$ are non-constant for all non-trivial peripheral elements $\alpha$ by \cite[Prop.\ 1.1.1]{CGLS}.
\end{proof}
%%%%%%%%%%%%%%%%%%%%

We can say more about the finitely many characters of reducible representations that lie on $\mathfrak{C}$. The following can be found in \cite[\S 6]{CCGLS}.

%%%%%%%%%%%%%%%%%%%%
\begin{prop}\label{CCGLS}
Let $K\subset S^3$ be a hyperbolic knot with Alexander polynomial $\Del_K(t)$ and let $\mathfrak{C} \subset X(K)$ be a canonical component. If $\rho:\Gam\rightarrow \SL_2(\C)$ is a reducible representation with $\chi_\rho\in \mathfrak{C}$, then the following hold.
\begin{enumerate}

\item There is a representation $\rho'$ with non-abelian image such that $\chi_\rho = \chi_{\rho'}\in \mathfrak{C}$.

\item If $\mu$ is a meridian of $K$, then $\rho(\mu)$ has an eigenvalue $z$ for which $z^2$ is a root of $\Delta_K(t)$.

\end{enumerate}
\end{prop}
%%%%%%%%%%%%%%%%%%%%

In fact all one needs for Proposition \ref{CCGLS}(1) to hold is for $\mathfrak{C}$ to contain the character of an irreducible representation. The key point is that if $t:R(K)\rightarrow X(K)$ is the map from the representation variety to the character variety defined in \S \ref{sec:BasicX} and if $\chi$ is the character of an irreducible representation, then $t^{-1}(\chi)$ is $3$-dimensional. On the other hand, if $\chi_\rho$ is the character of an abelian representation then $t^{-1}(\chi_\rho)$ is $2$-dimensional.

To prove part (2) of Proposition \ref{CCGLS}, it is shown in \S 6 of \cite{CCGLS} (following de Rham \cite{deRham}) that if $\mu_1,\ldots ,\mu_n$ is a collection of meridional generators for $\Gamma$, then the non-abelian representation $\rho'$ stated in Proposition \ref{CCGLS} can be described as follows. Recall that meridians in the knot group are are all conjugate, and hence have the same character values for all representations. Given this, there exist $w\in \C$ and $t_i\in \C$ for $i=1,\ldots ,n$ so that
\[
\rho'(\mu_i) = \begin{pmatrix} w & t_i\\ & \\ 0 & w^{-1} \end{pmatrix}.
\]
One then shows using the action of $\Gam$ on its commutator subgroup (e.g., see \cite[Ch.\ IV]{Neuwirth}) that $w^2=z$ is a root of $\Delta_K(t)$.

One consequence of Proposition \ref{CCGLS} is the following, where a \emph{parabolic representation} means a non-trivial representation $\rho:\Gam\rightarrow \SL_2(\C)$ all of whose non-trivial elements are parabolic. Note that this is equivalent to the statement that $\rho$ is a non-trivial representation for which $\chi_\rho(\gamma)=\pm 2$ for all $\gamma\in\Gamma$.

%%%%%%%%%%%%%%%%%%%%
\begin{cor}\label{nounipotent}
In the above notation, $\mathfrak{C}$ does not contain the character of a parabolic representation.
\end{cor}
%%%%%%%%%%%%%%%%%%%%

%%%%%%%%%%%%%%%%%%%%
\begin{proof}
We first show that if $\rho$ is a parabolic representation, then $\rho$ is abelian. To see see this, suppose that $a$ and $b$ are distinct meridians of $K$ for which $\rho(a)$ and $\rho(b)$ do not commute. The parabolic assumption allows us to conjugate $\rho$ such that
\[
\rho(a) = \begin{pmatrix} 1 & x \\ & \\ 0 & 1 \end{pmatrix}~\hbox{and}~\rho(b) = \begin{pmatrix} 1 & 0 \\ & \\ y & 1 \end{pmatrix}.
\]
Then $\tr(\rho([a,b])) = 2+x^2y^2$. We assumed this commutator is non-trivial and it is parabolic, and hence it has trace $\pm 2$.

When the trace is $2$, one of $x$ or $y$ is $0$, i.e., one of $\rho(a)$ or $\rho(b)$ is the identity. This contradicts the assumption that $\rho(a)$ and $\rho(b)$ do not commute. When the trace is $-2$, we have $x^2y^2=-4$. Conjugating by a diagonal matrix so $x = 1$, it follows that the product $\rho(ab)$ then has trace $2 \pm 2 i$, and hence is not parabolic, which is again a contradiction. Therefore, under any parabolic representation we deduce that all meridians must map to a common parabolic subgroup of $\SL_2(\C)$, and it follows that the image is abelian as required.

An abelian representation is reducible, so Proposition \ref{CCGLS} implies that there is a non-abelian representation $\rho'$ with $\chi_{\rho'}\in C$ and $\chi_\rho=\chi_{\rho'}$. Thus $\chi_{\rho'}(\gamma)=\pm 2$ for all $\gamma\in \Gamma$, i.e., $\rho'$ is also a parabolic representation, and hence $\rho'$ is abelian. This contradiction proves the corollary.
\end{proof}
%%%%%%%%%%%%%%%%%%%%

We record the following refinement of Lemma \ref{non-commuting} that will be helpful in the case where $\Gamma = \pi_1(S^3 \ssm K)$ for $K$ a hyperbolic knot.

%%%%%%%%%%%%%%%%%%%%
\begin{lem}\label{meridian_non-commuting}
Let $g$ and $h$ be distinct meridians of $K$ and $\chi_\rho\in \mathfrak{C}$ such that the restriction of $\rho$ to $\langle g, h \rangle$ is infinite and irreducible for some (hence any) representation $\rho$ with character $\chi_\rho$. If $\chi_\rho([g,h])=\pm2$, then $\rho([g,h])$ is a non-trivial parabolic element of $\SL_2(\C)$.
\end{lem}
%%%%%%%%%%%%%%%%%%%%

%%%%%%%%%%%%%%%%%%%%
\begin{proof}
Since $g$ and $h$ are meridians, they are conjugate in $\Gam$. In particular, $\chi_\rho(g) = \chi_\rho(h)$ for any $\rho \in R(\Gam)$. Suppose that $\chi_\rho([g,h]) = \pm 2$ but $\rho([g,h])$ is not parabolic, in which case either $\rho(g)$ and $\rho(h)$ commute or $\rho([g,h])$ is the negative of the identity matrix.

If $\rho(g)$ and $\rho(h)$ commute, we can conjugate $\rho$ such that
\begin{align*}
\rho(g) &= \begin{pmatrix} u & 1 \\ 0 & u^{-1} \end{pmatrix} \\
\rho(h) &= \begin{pmatrix} u & 0 \\ z & u^{-1} \end{pmatrix}
\end{align*}
One can then explicitly calculate $\rho([g,h])$ and see that the commutator is trivial if and only if $z = 0$ and $u = \pm 1$. However, $\Gam$ is normally generated by $h$, so the image of $\rho$ is either trivial or order $2$, and hence is not irreducible, which is a contradiction.

When $\rho([g,h])$ is the negative of the identity matrix, we similarly see that $z = 2$ and $u = \pm i$. This is conjugate to the representation
\begin{align*}
\rho'(g) &= \begin{pmatrix} i & 0 \\ 0 & -i \end{pmatrix} \\
\rho'(h) &= \begin{pmatrix} 0 & -1 \\ 1 & 0 \end{pmatrix},
\end{align*}
where $\langle g, h \rangle$ visibly has finite image. This contradiction completes the proof of the lemma.
\end{proof}
%%%%%%%%%%%%%%%%%%%%

%%%%%%%%%%%%%%%%%%%%
\section{Azumaya quaternion algebras and Brauer groups of Curves}\label{sec:Azumaya}
%%%%%%%%%%%%%%%%%%%%

In this section we recall some material concerning Azumaya algebras and Brauer groups of curves. Most of what we discuss is contained in Milne \cite{Mil}.

%%%%%%%%%%%%%%%%%%%%
\subsection{}\label{sec:AzumayaDef}
%%%%%%%%%%%%%%%%%%%%

Informally an Azumaya algebra is a generalization of a central simple algebra over a field $k$. To make this notion precise, we begin with the setting of an Azumaya algebra over a commutative local ring $R$ with residue field $k$. An algebra $A$ over $R$ is an \emph{Azumaya algebra} if $A$ is free of finite rank $r \geq 1$ as an $R$-module and $A\otimes k$ is a central simple algebra over $k$.

To define an Azumaya algebra over a curve we recall some additional terminology. Let $X$ be a Noetherian scheme. Recall that the \emph{structure sheaf} of $X$ is the sheaf of rings $\mathcal{O}_X$ such that for any open subset $U\subset X$, $\mathcal{O}_X(U)$ is the ring of regular functions on $U$. For each point $x\in X$, the \emph{stalk} of $\mathcal{O}_{X}$, denoted $\mathcal{O}_{X,x}$, is its local ring, i.e., the direct limit of $\mathcal{O}_X(U)$ over all open sets $U$ containing $x$. We denote the residue class field of $\mathcal{O}_{X,x}$ by $k(x)$.

A coherent sheaf of $\mathcal{O}_X$ modules $\mathcal{F}$ is a sheaf of abelian groups on $X$ such that, for any open subset $U\subset X$, $\mathcal{F}(U)$ is a finitely generated module over $\mathcal{O}_X(U)$ for which the module structure is compatible with restriction maps. The stalk of $\mathcal{F}$ at a point $x\in X$, denoted $\mathcal{F}_x$, is the direct limit of $\mathcal{F}(U)$ over those open sets $U$ containing $x$. One says that $\mathcal{F}$ is locally free if $\mathcal{F}_x$ is finitely generated and free over $\mathcal{O}_{X,x}$ for all $x \in X$.

An \emph{Azumaya algebra} $\mathcal{A}$ on $X$ is a locally free sheaf of $\mathcal{O}_X$ algebras such that $\mathcal{A}_x$ is an Azumaya algebra over the local ring $\mathcal{O}_{X,x}$ for every $x \in X$. Of particular interest to us are \emph{quaternion Azumaya algebras}, i.e., Azumaya algebras that are rank $4$ as locally free $\mathcal{O}_X$-modules.

%%%%%%%%%%%%%%%%%%%%
\subsection{}\label{sec:BrauerGp}
%%%%%%%%%%%%%%%%%%%%

Two Azumaya algebras $\mathcal{A}$ and $\mathcal{B}$ are \emph{equivalent} if there exist locally free sheaves of $\mathcal{O}_X$-modules $\mathcal{E}$ and $\mathcal{F}$ such that
\[
\mathcal{A}\otimes_{\mathcal{O}_X}\End_{\mathcal{O}_X}(\mathcal{E}) \cong \mathcal{B}\otimes_{\mathcal{O}_X}\End_{\mathcal{O}_X}(\mathcal{F}),
\]
where $\End_{\mathcal{O}_X}(\mathcal{H})$ is the sheaf of $\mathcal{O}_X$-module endomorphisms of an $\mathcal{O}_X$ module $\mathcal{H}$. This is an equivalence relation, and the group of equivalence classes of Azumaya algebras is called the \emph{Brauer group} of $X$, denoted by $\Br(X)$.

We now recall some basic results concerning Azumaya algebras and $\Br(X)$. For simplicity we restrict to $X$ that have properties of the kind that arise in our applications. For an abelian group $D$ and $n \ge 1$ in $\mathbb{Z}$, let $D[n]$ be the subgroup of elements with order dividing $n$. The following is an encyclopedia of classical facts about Azumaya algebras.

%%%%%%%%%%%%%%%%%%%%
\begin{thm}\label{thm:encyclo}
Suppose $X$ is a regular integral scheme of dimension at most $2$ with function field $K$ that is quasi-projective over a field or a Dedekind ring.
\begin{enumerate}

\item Descent theory gives a bijection between the set of isomorphism classes of Azumaya algebras $A_X$ of rank $n^2$ over $X$ and the elements of the \'etale \v{C}ech cohomology group $\hat{H}_{\textrm{\'{e}t}}^1(X, \PGL_n)$. Similarly, there is an isomorphism between isomorphism classes of rank $n$ locally free $O_X$-modules $E$ and elements of $\hat{H}_{\textrm{\'{e}t}}^1(X, \GL_n)$.

\item There is an exact sequence of \'etale sheaves of groups
\[
1 \to \mathbb{G}_m \to \GL_n \to \PGL_n \to 1
\]
on $X$. The cohomology of this sequence gives an exact sequence
\[
\hat{H}_{\textrm{\'{e}t}}^1(X, \GL_n) \to \hat{H}_{\textrm{\'{e}t}}^1(X, \PGL_n) \to \hat{H}_{\textrm{\'{e}t}}^2(X,\mathbb{G}_m)
\]
where
\[
\hat{H}_{\textrm{\'{e}t}}^2(X,\mathbb{G}_m) = H^2(X,\mathbb{G}_m) = \Br(X).
\]
With the notation of part (1), the isomorphism class of $E$ in $\hat{H}_{\textrm{\'{e}t}}^1(X, \GL_n)$ is identified in the above sequence with the class of $\End_{O_X}(E)$ in $\hat{H}_{\textrm{\'{e}t}}^1(X, \PGL_n)$. The isomorphism class of $A_X$ in $\hat{H}_{\textrm{\'{e}t}}^1(X, \PGL_n) $ is sent to the class $[A_X]$ of $A_X$ in $\Br(X)$.

\item Every element of $\Br(X)$ has finite order.

\item If $c$ is a class in $H^2(X,\mathbb{G}_m)$ of order $n$, then $c$ is represented by an Azumaya algebra $\mathcal{A}$ of rank $n^2$ over $X$.

\item The natural homomorphism $\Br(X) \to \Br(K)$ is injective. An Azumaya algebra $A_K$ over $K$ is determined up to isomorphism by its image in $\Br(K)$.

\item \label{AZOK} Let $x$ be a codimension one point of $X$, so $R = O_{X,x}$ is a discrete valuation ring with fraction field $K$. One says that an Azumaya algebra $A_K$ over $K$ extends over $x$ when there is an Azumaya algebra $\mathcal{A}_{R}$ over $R$ such that $A_K$ is isomorphic to $\mathcal{A}_R \otimes_R K$. This is the case for all codimension one points $x$ of $X$ if and only if $A_K$ extends to an Azumaya algebra $\mathcal{A}$ over $X$.

\item Suppose that $K$ has characteristic not equal to $2$. Every quaternion Azumaya algebra $A_K$ over $K$ is of the form
\[
A_K = \mathrm{Span}_K[1, I, J, IJ],
\]
where $I$ and $J$ are indeterminants for which there exist $\alpha, \beta \in K^*$ such that $I^2 = \alpha$, $J^2 = \beta$, and $IJ = -JI$. In other words, $A_K$ is the algebra with Hilbert symbol
\begin{align*}
\left(\frac{\alpha,\beta}{K}\right) \in H^2(\mathrm{Spec}(K),\{\pm 1\}) &= H^2(\mathrm{Spec}(K),\mathbb{G}_m)[2] \\
&= \Br(K)[2].
\end{align*}

\item Let $x$ be a codimension one point of $X$ and let $A_K$ be a quaternion Azumaya algebra over $K$. The tame symbol $\{\alpha, \beta\}_x$ of $A_K$ at $x$ is the class of
\[
(-1)^{\ord_x(\alpha) \ord_x(\beta)} \beta^{\ord_x(\alpha)}/\alpha^{\ord_x(\beta)}
\]
in $k(x)^*/(k(x)^*)^2$. If $k(x)$ has characteristic different from $2$, this symbol is trivial if and only if $A_K$ extends over $x$.

\end{enumerate}
\end{thm}
%%%%%%%%%%%%%%%%%%%%

%%%%%%%%%%%%%%%%%%%%
\begin{proof}
For statements (1)-(6), see Thm.\ IV.2.5, Thm.\ III.2.17, Prop.\ IV.2.7, Thm.\ IV.2.16, Cor.\ IV.2.6, and Remark IV.2.18(b) in \cite{Mil}, respectively. Statement (7) is shown in \cite[Prop.\ 4, \S 19.3]{Bourbaki}. Statement (8) is proven in the first four paragraphs of \cite[\S2]{ColliotEtAl}.
\end{proof}
%%%%%%%%%%%%%%%%%%%%

%%%%%%%%%%%%%%%%%%%%
\subsection{The proof of Theorem \ref{thm:FirstMain}}\label{sec:Harari}
%%%%%%%%%%%%%%%%%%%%

We now give the proof of Theorem \ref{thm:FirstMain}. First, we recall our assumptions. Let $\Gam$ be a finitely generated group. Suppose that $k$ is a number field realized as a subfield of $\C$ under a fixed embedding and that $C$ is a geometrically irreducible curve over $k$ that is a closed subscheme of $X(\Gamma)_k = X(\Gam)_\Q \otimes_\Q k$ such that $C(\C)$ contains the character of an irreducible representation. As in the statement of the theorem, $C^\sharp$ denotes the normalization of $C$ and $\wt{C}$ is the unique smooth projective curve over $k$ birational to $C$. Let $k(C)$ be the common function field of $C$, $C^\sharp$, and $\wt{C}$.

From Lemma \ref{non-commuting} and Corollary \ref{hilberttaut} we can find elements $g, h \in \Gam$ such that the canonical quaternion algebra $A_{k(C)}$ over $k(C)$ is well defined and has Hilbert symbol
\begin{equation}
\label{eq:nicely}
\left( \frac{I_g^2 - 4\, ,\, I_{[g,h]}-2}{k(C)} \right).
\end{equation}
This gives Theorem \ref{thm:FirstMain}(1). Theorem \ref{thm:FirstMain}(2) follows from Lemma \ref{lem:ResidueFieldTraceField}.

To prove part (3) of Theorem \ref{thm:FirstMain}, suppose that $A_{k(C)}$ does not extend to define an Azumaya algebra over all $\wt{C}$. We can find an open affine subset $U$ of smooth points of $C\subset \wt{C}$ over which $A_{k(C)}$ extends to an Azumaya algebra $A_U$ having Hilbert symbol as in Equation \eqref{eq:nicely} at every point. In view of Lemma \ref{taut_irred}, we can furthermore require that the points of $U$ correspond to absolutely irreducible representations of $\Gam$ and that $U$ avoids any prescribed finite set of closed points.

For each place $v$ of $k$, let $s_v:U(k_v) \to \Br(k_v)$ be the map defined by specializing $A_U$ at a point of $U(k_v)$. By a theorem of Harari \cite[Thm.\ 2.1.1]{Harari}, there are infinitely many places $v$ of $k$ such that $s(x_v)$ is non-trivial for some $x_v \in U(k_v)$. By continuity of Hilbert symbols, we can then find a $v$-adic disk $B_v$ of positive radius around $x_v$ in $U(k_v)$ such that $s_v$ is non-trivial at all elements of $B_v$.

It is a theorem of Rumely \cite[Thm.\ 0.3]{Rumely2} that there are infinitely many algebraic points $w$ in $U(\overline{k})$ that have all of their conjugates over $k$ contained in $B_v$. The specialization of $A_U$ at such a $w$ is thus a quaternion algebra over $k(w)$ whose local invariants at places of $k(w)$ over $v$ are non-trivial. This proves Theorem \ref{thm:FirstMain}(3).

We now prove Theorem \ref{thm:FirstMain}(4). For this, suppose that $A_{k(C)}$ extends to an Azumaya algebra $A_{\wt{C}}$ over all of $\wt{C}$. Suppose $z \in C^\sharp(\overline{k})$ and that $w = w(z) \in C(\overline{k})$ corresponds to the character of an absolutely irreducible representation $\rho = \rho_w$ of $\Gamma$. By the `r\'esultat classique' mentioned in the second remark after the statement of \cite[Thm.\ 2.1.1]{Harari}, it will suffice to show that the specialization $A_{\wt{C}} \otimes k(z)$ of $A_{\wt{C}}$ at $z$ has the same class in the Brauer group $\Br(k(z))$ as $A_\rho \otimes_{k_\rho} k(z)$.

Recall that $w \in C(\overline{k})$ implies that all of the character functions defining the embedding of $C$ into $X(\Gam)_k$ are regular at $w$. By Lemma \ref{taut_irred}, the tautological representation $P_C:\Gamma \to \mathrm{SL}_2(k(D))$ has image contained in the $k(D)$-span of four elements $\{P_C(g_j)\}_{j = 1}^4$ that are linearly independent over $k(D)$. For an arbitrary $\gamma \in \Gamma$, we can determine the coefficients $a_j \in k(D)$ such that $P_C(\gamma) = \sum_j a_j P_C(g_j)$ from
the equations
\[
\mathrm{Tr}(P_C(g_i \gamma)) = \sum_{j=1}^4 a_j \mathrm{Tr}(P_C(g_i g_j)),
\]
since the trace gives a non-degenerate pairing from $\M_2(k(D))$ to $k(D)$.

Now, using the fact that $\mathrm{Tr}(\tau)$ lies in the local ring $O_{C,w}$ for all $\tau \in \Gamma$, we see that the $O_{C,w}$-subalgebra $A_w$ of $\M_ 2(k(D))$ generated by $P_C(\Gamma)$ is finitely generated over $O_{C,w}$. Furthemore, $A_w \otimes_{O_{C,w}} k(C) = A_{k(C)}$ and $A_w \otimes_{O_{C,w}} k(w) = A_\rho$. Since $A_\rho$ is a quaternion algebra, this means by definition that $A_w$ is a quaternion Azumaya algebra over $O_{C,w}$. We conclude that $A_z = A_w \otimes_{O_{C,w}} O_{C^\sharp,z}$ is a quaternion Azumaya algebra over $O_{C^\sharp,z}$ with generic fiber $A_{k(C^\sharp)} \cong A_{k(C)}$. The localization $A_{\wt{C},z}$ of $A_{\wt{C}}$ at $z$ is also such an Azumaya algebra, so $A_z$ and $A_{\wt{C},z}$ have the same
class in the Brauer group $\Br(O_{C^\sharp,z})$ since this Brauer group injects into $\Br(k(C))$ by Theorem \ref{thm:encyclo}(5). Therefore $A_z \otimes k(z) = A_{\rho} \otimes_{k_\rho} k(z)$ and $A_{\wt{C}} \otimes k(z)$ have the same class in $\Br(k(z))$, as required.

Part (5) follows from parts (1), (2), and (5) of Theorem \ref{thm:encyclo}. \qed

%%%%%%%%%%%%%%%%%%%%
\section{Azumaya algebras and canonical components}\label{sec:Proofs}
%%%%%%%%%%%%%%%%%%%%

We now specialize the above to the case of most interest to us. Let $M={\bf H}^3/\Gamma$ be a 1-cusped finite volume hyperbolic 3-manifold and $\mathfrak{C} \subset X(\Gam)$ be a canonical component over the complex numbers. Define $k$ to be the the field of constants of $\mathfrak{C} \subset X(\Gamma)$, and let $C \subset X(\Gam)_k$ be the the canonical component over $k$, and note that $k$ is a number field. We will be particularly interested in the case when $M=S^3\ssm K$ for $K$ a hyperbolic knot.

%%%%%%%%%%%%%%%%%%%%
\subsection{}
%%%%%%%%%%%%%%%%%%%%

When $C$ is the canonical component, that $A_{k(C)}$ is a quaternion algebra follows from Lemma \ref{taut_faithful}. Indeed, Lemma \ref{taut_faithful} implies that we can apply Corollary \ref{hilberttaut} to $C$ to describe a Hilbert symbol for $A_{k(C)}$. The challenge is to now determine when $A_{k(C)}$ can be extended globally to define a quaternion Azumaya algebra $A_{\wt{C}}$ on the smooth projective model $\wt{C}$ of $C$. To that end, the remainder of this section aims at understanding when this happens, in particular proving Theorem \ref{thm:MainKnot}.

\medskip

For emphasis, in the remainder of this section we have:

\medskip

\noindent
\textbf{Assumptions:}~We fix $M=S^3\ssm K$, where $K$ is a hyperbolic knot, $\Gam=\pi_1(S^3\ssm K)$ and $C\subset X(K)_k$ the canonical component over the field of constants $k$.

\medskip

\noindent
In the next two subsections we prove the following results. Taken together, these will complete the proof of Theorem \ref{thm:MainKnot}. In \S\ref{sec:HeusenerPorti}, we consider the converse.

%%%%%%%%%%%%%%%%%%%%
\begin{prop}\label{prop:IdealExtension}
The quaternion algebra $A_{k(C)}$ extends to an Azumaya algebra over the (Zariski open) set of points $\chi \in \wt{C}$ where:
\begin{enumerate}

\item $\chi=\chi_\rho$ is the character of an irreducible representation of $\Gamma$;

\item $\chi\in \mathcal{I}(\wt{C})$ is an ideal point.

\end{enumerate}
\end{prop}
%%%%%%%%%%%%%%%%%%%%

To be precise, we say that $\chi \in \wt{C}$ is irreducible (resp.\ reducible) if the image of $\chi$ on $C$ under the rational map $\wt{C} \to C$ described in \S \ref{ssec:TreeAction} has image the character of an irreducible (resp.\ reducible) representation. Recall that the ideal points $\mathcal{I}(\wt{C})$ are the set of points on $\wt{C}$ where this rational map is not well-defined.

%%%%%%%%%%%%%%%%%%%%
\begin{lem}\label{lem:reducible}
Suppose that the Alexander polynomial $\Del_K(t)$ satisfies property $(\star)$ of Theorem \ref{thm:MainKnot}. Then at any point $\chi_\rho\in \wt{C}$ that is the character of a reducible representation $\rho$ we have that $A_{k(C)}$ extends over $\chi_\rho$.
\end{lem}
%%%%%%%%%%%%%%%%%%%%

%%%%%%%%%%%%%%%%%%%%
\subsection{Proof of Proposition \ref{prop:IdealExtension}}
%%%%%%%%%%%%%%%%%%%%

Lemmas \ref{non-commuting} and \ref{meridian_non-commuting} along with Corollary \ref{hilberttaut} imply that we can choose a pair of non-commuting meridians $g,h$ in $\Gam$, so that the functions $f_g = I_g^2 - 4$ and $I_{[g,h]}-2$ can be used to describe a Hilbert symbol for $A_{k(C)}$. Note that $I_{[g,h]}-2$ cannot be identically $0$ on $C$ since it is non-zero at the character of the discrete and faithful representation.

\medskip

\noindent
\textbf{Notation:}~{\em For $f\in k(C)$, let $Z(f)$ denote the set of zeroes of $f$ in $\wt{C}$.}

\medskip

Let $W \subseteq \wt{C}$ be $\mathcal{I}(\wt{C})$ together with $Z = Z(f_g)\cup Z(I_{[g,h]}-2)$ on $\wt{C}$. By Lemma \ref{finitelymanyreducible}, $W$ is a finite collection of points that includes the set $C_R$ of reducible characters on $\wt{C}$. Note also that any poles of $f_g$ and $I_{[g,h]}-2$ occur at points in $\mathcal{I}(\wt{C}) \subseteq W$.

Given this, for any point in the Zariski open set $U=\wt{C}\smallsetminus W$ we have that $\ord_P(f_g)=\ord_P(I_{[g,h]}-2)=0$. It follows that the tame symbol $\{f_g,I_{[g,h]}-2\}_P$ is trivial, and therefore $A_{k(C)}$ can be extended over $U$.

To see this directly from the definition, let $V \subset \wt{C} \ssm \mathcal{I}(\wt{C})$ be the points associated with characters of irreducible representations. For any $P \in V \ssm C_R$ with image on $C$ the character $\chi_\rho$ of an irreducible representation $\rho$, we can choose $g'$ and $h'$ in $\Gam$ such that the elements $\{1, \rho(g'), \rho(h'), \rho(g'h')\}$ form a $k_\rho$-basis for $A_\rho$ for some (hence any) representation with character $\chi_\rho$. Then $\{1, P_C(g'), P_C(h'), P_C(g'h')\}$ is a basis for $A_{k(C)}$. Moreover, if $\mathcal{O}_P$ is the local ring of $P$, the $\mathcal{O}_P$-span of this basis defines an Azumaya algebra $A_P$ over $\mathcal{O}_P$. Indeed, the reduction of $A_P$ modulo the maximal ideal of $\mathcal{O}_P$ is the given basis for the quaternion algebra $A_\rho$ (note that the residue field of the point $P$ is generated by $k_\rho$ and the field of constants $k$ by Lemma \ref{lem:ResidueFieldTraceField}), so $A_P \otimes k(P)$ is a central simple algebra. Thus $A_{k(C)}$ extends over $P$ by Theorem \ref{thm:encyclo}(6).

\medskip

We now show how one extends the Azumaya algebra $A_{k(C)}$ over points in $\mathcal{I}(\wt{C})$. Fix $P\in \mathcal{I}(\wt{C})$ and denote the local ring at $P$ by $\mathcal{O}_P$ and its maximal ideal by $\mathfrak{m}_P$. Note that, from the discussion at the end of \S \ref{ssec:TreeAction}, if $\wt{I}_g$ is in $\mathcal{O}_P$ for all $g$ in $\Gamma$, then $P \in C^\#$. Since $P$ is an ideal point, we can find $g\in \Gamma$ such that $\wt{I}_g \notin \mathcal{O}_P$. Note that the element $g$ need not be a meridian in this case, nor even a peripheral element; by \cite{CullerShalen} we can take $g$ to be peripheral when the incompressible surface detected by the ideal point is not closed.

Regardless, Lemma \ref{non-commuting} allows us to use $g$ as part of a basis for $A_{k(C)}$. As before, since $C$ is the canonical component, Lemma \ref{finitelymanyreducible} shows that there are only finitely many characters of reducible representations on $C$, and only finitely many places where $\wt{f}_g$ takes on the value $0$, since $\wt{I}_g$ is non-constant by assumption. Thus we can take $\wt{f}_g$ as one term in a Hilbert symbol
\[
\left(\frac{\wt{f}_g,\beta}{k(C)}\right)
\]
for the canonical quaternion algebra $A_{k(C)}$, where $\beta\in k(C)$ is constructed using Lemma \ref{non-commuting}.

We assumed that the order $\ord_P(\wt{f}_g)$ is negative. Since changing a term in a Hilbert symbol by the square of an element of $k(C)$ does not change the resulting quaternion algebra, the element
\[
\alpha_g = \frac{\wt{f}_g}{I_g^2}
\]
defines another Hilbert symbol
\[
\left(\frac{\alpha_g,\beta}{k(C)}\right)
\] 
for $A_{k(C)}$ as a quaternion algebra over $k(C)$. Note that $\alpha_g$ is a unit in $\mathcal{O}_P$, and the image of $\alpha_g$ in the residue field $k(P) = \mathcal{O}_P/\mathfrak{m}_P$ of $P$ is $1$. To see this observe that
\[
\alpha_g = \frac{I_g^2 - 4}{I_g^2} = 1 - \frac{4}{I_g^2},
\]
and then the rest is clear from $\ord_P(4/I_g^2) = -2\, \ord_P(I_g) > 0$.

To determine whether $A_{k(C)}$ extends to an Azumaya algebra at $P$, we need to show that the tame symbol $\{\alpha_g, \beta\}_P$ is trivial. To prove this, let
\begin{align*}
s &= \ord_P(\alpha_g) \\
r &= \ord_P(\beta).
\end{align*}
Then the tame symbol $\{\alpha_g, \beta\}_P$ is the image in $k(P)^* / (k(P)^*)^2$ of
\[
(-1)^{r s} \frac{\alpha_g^r}{\beta^s}.
\]
However, we saw that $s = 0$, so this simplifies to just the image in $k(P)^*$ of $\alpha_g^r$. Since $\alpha_g$ has image $1$ in $k(P)^*$, the tame symbol is therefore trivial and we can extend $A_{k(C)}$ over the ideal point $P$. This proves that $\mathcal{A}_{\wt{C}}$ is also defined at points of $\mathcal{I}(\wt{C})$ as required, and hence completes the proof of Proposition \ref{prop:IdealExtension}.\qed

%%%%%%%%%%%%%%%%%%%%
\begin{rem}
The careful reader will notice that, when $P \in \wt{C}$ lies over a singular point on the affine curve $C$, we must calculate in the discrete valuation ring $\mathcal{O}_P$ for `$\ord_P$' to even make sense. When $x$ is a smooth point of $C$, one has that $\mathcal{O}_x \cong \mathcal{O}_P$ and there is no difference. When $x$ is a singular point, $\mathcal{O}_P$ might be a bigger ring, but if we can extend our Azumaya algebra on $C$ over $x$, then we can certainly extend it over $P$. However, the converse is not necessarily true: it is possible that the Azumaya algebra does not extend over a singular point on the affine curve, but does extend over any point above it on the smooth projective model. This subtlety will arise in our converse to Theorem \ref{thm:MainKnot} in \S \ref{sec:HeusenerPorti}.
\end{rem}
%%%%%%%%%%%%%%%%%%%%

%%%%%%%%%%%%%%%%%%%%
\subsection{}
%%%%%%%%%%%%%%%%%%%%

We now give the proof of Lemma \ref{lem:reducible}. This, with Proposition \ref{prop:IdealExtension}, completes the proof of Theorem \ref{thm:MainKnot}.

We begin with some preliminary comments. Let $P \in \wt{C}$ lie above $\chi_\rho \in C$ for $\rho$ a reducible representation. We can assume that $\rho$ is non-abelian by Proposition \ref{CCGLS}, and hence it is conjugate into the group of upper-triangular matrices but is not parabolic. Indeed, recall from the discussion after Proposition \ref{CCGLS} that if $\mu_1,\ldots ,\mu_n$ is a collection of meridional generators for $\Gamma$ then
\[
\rho(\mu_i) = \begin{pmatrix} w & t_i\\ & \\ 0 & w^{-1} \end{pmatrix},
\]
where $w^2=z$ is a root of $\Delta_K(t)$, $t_i\in \C$ for $i=1,\ldots ,n$. From condition $(\star)$, we have $\Q(w)=\Q(w+w^{-1})$. If $k$ is the field of constants of $C$, then we also have that $k(w)=k(w+w^{-1})$. In particular, the residue field $k(P)$ of $P$ satisfies $k(P)=k(w+w^{-1})=k(w)$.

Assume by way of contradiction that $A_{k(C)}$ does not extend to an Azumaya algebra $A_{\wt{C}}$ at $P$. Using Lemma \ref{meridian_non-commuting} we can choose meridians $\mu$ and $\nu$ so that $A_{k(C)}$ is defined by the Hilbert symbol
\[
\left( \frac{a,b}{k(C)} \right),
\] 
where $a=f_\mu$, $b=I_{[\mu,\nu]}-2$. Since we assumed that $A_{k(C)}$ does not extend, the tame symbol $\{a,b\}_P$ must be non-trivial.

Corollary \ref{nounipotent} implies that $\ord_P(a)=0$, and $\ord_P(b)>0$ by Lemma \ref{trace2}. If $\ord_P(b)$ is even, then the tame symbol is trivial, and since we are assuming this is not the case we have that $\ord_P(b)$ is odd. Furthermore, after dividing $b$ by a square in $k(C)$, we can assume that $\ord_P(b) = 1$. Therefore the tame symbol is just the class $a'$ of $a$ in $k(P)^*/(k(P)^*)^2$, and in particular $a'$ cannot be a square in $k(P)$.

On the other hand, evaluating at $P$, $a'$ is the class of
\[
(w+w^{-1})^2-4 = (w-w^{-1})^2,
\]
since $w\in k(w+w^{-1})$ by assumption. In particular, $a'$ is a square in the residue class field $k(P)$, which implies that the tame symbol is trivial. This contradiction proves Lemma \ref{lem:reducible}.\qed

%%%%%%%%%%%%%%%%%%%%
\begin{rem}\label{notS3} 
Note that the arguments of \S 4.2 apply more generally. In particular if $Y$ is a closed orientable 3-manifold and $K\subset Y$ is a knot with hyperbolic complement, trivial Alexander polynomial, and $H_1(Y\ssm K, \Z)\cong \Z$, then the canonical quaternion algebra can be used to define an Azumaya algebra over all points of the smooth model of the canonical component. Note that \cite[\S4]{HP} describes a generalization of De Rham's result on characters of reducible representations, and in particular, triviality of the Alexander polynomial excludes there being non-abelian reducible representations to consider.
\end{rem}
%%%%%%%%%%%%%%%%%%%%

We now discuss the extent to which condition $(\star)$ is almost an if and only if statement.

%%%%%%%%%%%%%%%%%%%%
\subsection{The converse to Theorem \ref{thm:MainKnot}}\label{sec:HeusenerPorti}
%%%%%%%%%%%%%%%%%%%%

The primary goal of this section is to prove the following proposition.

%%%%%%%%%%%%%%%%%%%%
\begin{prop}\label{prop:StarConverse}
Let $K$ be a hyperbolic knot in $S^3$. Let $C \subset X(K)_k$ be a canonical component over the number field $k$. Suppose that $\chi_\rho$ is a smooth point on $C$ for each character $\chi_\rho \in C$ of a non-abelian reducible representation. Then $A_{k(C)}$ extends to define an Azumaya algebra on $\wt{C}$ if and only if $k(w)$ equals $k(w + w^{-1})$ for every square root $w$ of a root of $\Del_K(t)$ associated with a non-abelian character on $\wt{C}$.
\end{prop}
%%%%%%%%%%%%%%%%%%%%

When $k = \Q$, note that the condition $k(w) = k(w + w^{-1})$ just becomes condition $(\star)$ from Theorem \ref{thm:MainKnot}. According to work of Heusener--Porti--Su\'arez Peir\'o \cite[Thm.\ 1.1]{HP}, the point $\chi_\rho$ is always smooth when the associated root $z$ of the Alexander polynomial is simple, i.e., has multiplicity one. In particular, we see that the full converse to Theorem \ref{thm:MainKnot} is quite often also true. We will prove the following after we prove Proposition \ref{prop:StarConverse}.

%%%%%%%%%%%%%%%%%%%%
\begin{thm}\label{thm:BigConverse}
Let $K$ be a hyperbolic knot and let $C\subset X(K)_k$ be a canonical component of its $\SL_2$ character variety. Assume that the inverse image $C^\prime$ of $C$ in $X(K)_\C$ is the unique component of $X(K)_\C$ containing the character of an irreducible representation. If all roots of the Alexander polynomial $\Del_K(t)$ are simple, then $A_{k(C)}$ extends to define an Azumaya algebra over the entire smooth projective model of $C$ if and only if condition $(\star)$ of Theorem \ref{thm:MainKnot} holds for every root of $\Del_K(t)$.
\end{thm}
%%%%%%%%%%%%%%%%%%%%

We also briefly note that large families of knots, like all twist knots, have Alexander polynomial with only simple roots (see \S 6).

%%%%%%%%%%%%%%%%%%%%
\begin{proof}[Proof of Proposition \ref{prop:StarConverse}]
Let $P = \chi_\rho \in C$ be the character of a non-abelian reducible representation $\rho$, $z$ be the associated root of the Alexander polynomial, and $w$ be a square root of $z$. Recall that we have non-commuting elements $g, h \in \Gam$ such that our Azumaya algebra $A_{k(C)}$ over the function field $k(C)$ of $C$ has Hilbert symbol
\[
A_{k(C)} = \left(\frac{I_g^2 - 4\, ,\, I_{[g, h]} - 2}{k(C)}\right).
\]
Assuming that $k(w)$ is not equal to $k(w + w^{-1})$ for this root of the Alexander polynomial we will prove that that $A_{k(C)}$ does not extend over $P$. The converse was already proved in Theorem \ref{thm:MainKnot}, so this suffices to prove the proposition.

Define $\alpha = I_g^2 - 4$. Taking $g$ to be a meridian of our knot, we claim that $\alpha(P) \neq 0$. Indeed, since $\rho$ is reducible it is conjugate into upper-triangular matrices and since $\Gam$ is generated by meridians, if $\alpha(P) = 0$ one sees that $\rho$ is a parabolic representation where the image of each meridian has trace $2$. Corollary \ref{nounipotent} implies that $\rho$ is not parabolic, hence $I_g$ cannot take the value $\pm 2$ at $P$, and so $\alpha$ cannot be zero at $P$. This proves the claim.

In other words, $\ord_P(\alpha) = 0$. Since $\rho$ is reducible, $(I_{[g, h]}-2)(P) = 0$, i.e., $\ord_P(I_{[g, h]} - 2) > 0$. Scaling by squares in $k(C)$, we can replace $I_{[g,h]}-2$ with a function $\beta \in k(C)$ such that $\ord_P(\beta) \in \{0, 1\}$. This means that the tame symbol for $A_{k(C)}$ at $P$ becomes
\[
\{\alpha, \beta\}_P = \alpha^{\ord_P(\beta)}.
\]
We must prove that the tame symbol is non-trivial. In other words, we must show that $\alpha$ is not a square in $k(P)$ and $\ord_P(\beta) = 1$.

Since $k(w)$ clearly has degree either one or two over $k(w + w^{-1})$, we conclude that $k(w) / k(w + w^{-1})$ is degree exactly two. Note that $k(P)$ is a number field. In fact, Lemma \ref{lem:ResidueFieldTraceField} implies that $k(P)$ is the subfield of $\C$ generated over $k$ by the values of the character $\chi_\rho$. From the discussion following Proposition \ref{CCGLS}, we see that $k(P) = k(w + w^{-1})$.

Let $I_P$ be a square root of $\alpha(P) \in k(P)^*$. We claim that $k(P)(I_P)$ equals $k(w)$. Indeed, $I_P$ is a square root of
\[
(w+w^{-1})^2 - 4 = w^2 + w^{-2} - 2 = (w - w^{-1})^2,
\]
which does not lie in $k(w + w^{-1})$ by our assumption that $k(w)$ is degree two over $k(w + w^{-1})$. Therefore, $k(P)(I_P)$ is a quadratic extension of $k(P)$ and so $\alpha$ is not a square in $k(P)^*$ as claimed.

We now must show that $\ord_P(\beta) = 1$. Assuming that the root of the Alexander polynomial associated with $P$ is a simple root, the proof of \cite[Lem.\ 5.7]{HP} proves that $-\beta / \alpha$ is a local coordinate in a sufficiently small analytic neighborhood of the smooth point $P$. Since $\ord_P(\alpha) = 0$, this implies that $\ord_P(\beta) = 1$, as desired. This completes the proof.
\end{proof}
%%%%%%%%%%%%%%%%%%%%

We will show that Theorem \ref{thm:BigConverse} is a consequence of Proposition \ref{prop:StarConverse} and the following lemma.

%%%%%%%%%%%%%%%%%%%%
\begin{lem}\label{lem:UniqueToQ}
Let $K$ be a hyperbolic knot and let $C\subset X(K)_k$ be a canonical component of its $\SL_2$ character variety. Assume that the inverse image $C^\prime$ of $C$ in $X(K)_\C$ is the unique component of $X(K)_\C$ containing the character of an irreducible representation. Then $C$ has field of constants $\Q$.
\end{lem}
%%%%%%%%%%%%%%%%%%%%

%%%%%%%%%%%%%%%%%%%%
\begin{proof}
Let $\pi : X(K)_k \to X(K)_\Q$ be the projection. Then $\pi(C)$ is an irreducible subscheme of $X(K)_\Q$ with some field of constants $k_0$. The inverse image of $\pi(C)$ under the natural projection $X(K)_\C \to X(K)_\Q$ contains $[k_0 : \Q]$ conjugates of $C^\prime$ under the action of $\Gal(\C/\Q)$ on the second factor of $X(K)_\C = X(K)_\Q \otimes_\Q \C$. Indeed, if $B$ is a $k_0$-algebra, then $B \otimes_\Q \C$ is a direct sum of $B \otimes_{k_0, \sigma} \C$ over all embeddings $\sigma$ of $k_0$ into $\C$.

The representations associated with points on these conjugates of $C^\prime$ are $\Gal(\C / \Q)$-conjugates of representations associated with points on $C^\prime$. In particular, all conjugates of $C^\prime$ contain the character of an irreducible representation. Since $C^\prime$ is the only component of $X(K)_\C$ containing the character of an irreducible representation, we must have $[k_0 : \Q] = 1$, i.e., $k_0 = \Q$, as claimed.
\end{proof}
%%%%%%%%%%%%%%%%%%%%

%%%%%%%%%%%%%%%%%%%%
\begin{proof}[Proof of Theorem \ref{thm:BigConverse}]
As discussed above, it suffices by Proposition \ref{prop:StarConverse} to show that the canonical component has field of constants $\Q$. By hypothesis, this follows from Lemma \ref{lem:UniqueToQ}.
\end{proof}
%%%%%%%%%%%%%%%%%%%%

%%%%%%%%%%%%%%%%%%%%
\section{Integral models and Theorems \ref{thm:MainKnotIntegral} and \ref{non-trivial}}\label{sec:IntegralModel}
%%%%%%%%%%%%%%%%%%%%

The purpose of this section is to prove Theorems \ref{thm:MainKnotIntegral} and \ref{non-trivial}. The proof of Theorem \ref{thm:MainKnotIntegral} requires revisiting some of our previous calculations in the case where the residue class fields are finite, which introduces particular difficulties in characteristic $2$. The framework that allows us to do this is that of integral models and we refer the reader to \cite{ChinburgMinimal}, \cite{Lichtenbaum}, and \cite{Liu} for further details. The proof of Theorem \ref{non-trivial} uses Theorem \ref{thm:MainKnotIntegral} as well as the Tate--Shafarevich group and a relative version of it due to Stuhler (see \cite{Stuhler} for further details).

%%%%%%%%%%%%%%%%%%%%
\subsection{Ramification at codimension one points.}\label{s:ramification}
%%%%%%%%%%%%%%%%%%%%

Suppose that $O$ is a Dedekind ring of characteristic $0$ with fraction field ${k}$ and that $C$ is a regular projective curve over ${k}$. By enlarging ${k}$ if necessary we can assume that $C$ is geometrically irreducible over ${k}$. In other words, by passing to a certain extension and taking an irreducible component, we can assume that $C$ remains irreducible in an algebraic closure of ${k}$. The theory of integral models implies that there is a regular projective curve $\mathcal{C}$ over $O$ such that $C$ is isomorphic to $\mathcal{C} \otimes_O {k}$.

Such $\mathcal{C}$ are not unique. However, there are always $\mathcal{C}$ that are relatively minimal in the sense that any proper morphism $\mathcal{C} \to \mathcal{C}'$ to another regular projective model $\mathcal{C}'$ of $C$ must be an isomorphism. If $C$ has positive genus, then all relatively minimal models are isomorphic. Finally, suppose $\mathcal{D}$ is any regular projective scheme over $O$ whose function field $k(\mathcal{D})$ is isomorphic to $k(C)$. Then the general fiber $\mathcal{D} \otimes_O {k}$ is a regular projective curve over ${k}$ with function field $k(C)$. This forces $\mathcal{D} \otimes_O {k}$ to be isomorphic to $C$ over ${k}$.
 
Let $A_{k(C)}$ be a quaternion algebra over the function field $k(C)$ of $C$. Theorem \ref{thm:encyclo}(\ref{AZOK}) implies that $A_{k(C)}$ extends to an Azumaya algebra $A_{\mathcal{C}}$ over $\mathcal{C}$ if an only if it extends over every codimension one point $P$ of $\mathcal{C}$. The latter condition means that there is an Azumaya algebra $A_P$ over the local ring $O_P = O_{\mathcal{C},P}$ of $P$ such that $A_P \otimes_{O_P} k(C) \cong A_{k(C)}$. Here $O_P$ is a discrete valuation ring and the residue field $k(P)$ is a global field, since $\mathcal{C}$ is a scheme of dimension $2$. Let $\hat{O}_P$ be the completion of $O_P$ and $k(C)_P$ be the fraction field of $\hat{O}_P$.

The next result follows from \cite[Lem.\ 3.4]{ParimalaSuresh}. For convenience, we note that our $k(C)$, $O_P$, and $A_{k(C)}$ correspond to their $k$, $R$, and $D$ respectively.

%%%%%%%%%%%%%%%%%%%%
 \begin{lem}\label{lem:parnice}
The quaternion algebra $A_{k(C)}$ extends over $P$ if and only if the quaternion algebra $A_{k(C)_P}=A_{k(C)} \otimes_{k(C)} k(C)_P$ over $k(C)_P$ extends over the maximal ideal $\hat{P}$ of $\hat{O}_P$, i.e., there is an Azumaya algebra $\mathcal{D}$ over $\hat{O}_P$ so that $\mathcal{D} \otimes_{\hat{O}_P} k(C)_P$ determines the same class as $A_{k(C)}$ in the Brauer group of $k(C)_P$.
\end{lem}
%%%%%%%%%%%%%%%%%%%%
 
We will also need the following weak sufficient criterion for $A_{k(C)}$ to extend over $P$. Note that in the present situation we allow for the possibility that $k(C)$ has characteristic zero (so $2 \neq 0$) but the residue class fields can be finite (and in particular characteristic $2$). As before,
we use $\ord_P$ to denote the order of a zero or pole.

%%%%%%%%%%%%%%%%%%%%
\begin{prop}\label{prop:Parimala-Suresh}
Let $\alpha,\beta \in k(C)$ define a Hilbert symbol $\{\alpha,\beta\}$ for $A_{k(C)}$ for some $\alpha, \beta \in k(C)$. For a given $P\in C$ assume that
\[
\ord_P(1-\alpha) > 2\, \ord_P(2).
\]
Then $A_{k(C)}$ extends over $P$.
\end{prop}
%%%%%%%%%%%%%%%%%%%%

%%%%%%%%%%%%%%%%%%%%
\begin{proof}
By Lemma \ref{lem:parnice}, it is sufficient to show that the Hilbert symbol defined by $\{\alpha,\beta\}$ over $k(C)_P$ defines a matrix algebra. The result then follows from Hensel's lemma, since $\alpha$ is a square in $\hat{O}_P$ when
\[
\mathrm{ord}_P(y^2 - \alpha) > \mathrm{ord}_P((2y)^2) = 2\, \mathrm{ord}_P(2 y)
\]
for some $y \in k(C)_P$, and our assumption allows us to take $y = 1$.
\end{proof}
%%%%%%%%%%%%%%%%%%%%

In the complete local case, we also need the following result concerning the structure of maximal orders in quaternion division algebras.

%%%%%%%%%%%%%%%%%%%%
\begin{prop}\label{prop:structure}
Suppose that $O$ is a complete discrete valuation ring with fraction field $F$ of characteristic $0$. Let $A$ be a quaternion algebra over $F$ that does not extend to an Azumaya algebra over $O$, and let $n:A \to F$ be the reduced norm.
\begin{enumerate}

\item[i.] Any such $A$ is a division ring, and the set $D$ of elements $\alpha \in A$ such that $n(\alpha) \in O$ is the unique maximal $O$-order in $A$.

\item[ii.] There is an element $\lambda \in D$ such that $J = D \lambda$ is the unique maximal two sided ideal of $D$. All non-zero two-sided ideals of $D$ are powers of $J$.

\item[iii.] Let $\pi$ be a uniformizer in $O$. Then $J^2 = D \pi$ and $D/J$ is a quadratic extension field of the residue field $k = O/O\pi$ of $O$.

\item[iv.] Suppose $z \in D$ has $n(z) = 1$. Let $\tilde{z}$ be the image of $z$ in $D/J$, and suppose that $\tilde{z}$ is quadratic over $k$, so $k(\tilde{z}) = D/J$. Then $k(\tilde{z})$ is separable over $k$, $J/J^2$ is a one-dimensional $k(\tilde{z})$-vector space, and the conjugation action of $z$ on this space is given by left multiplication by $\tilde{z}^2$. The conjugation action of $z$ on $D/J$ is trivial.

\end{enumerate}
\end{prop}
%%%%%%%%%%%%%%%%%%%%

%%%%%%%%%%%%%%%%%%%%
\begin{proof}
We know $A$ is a division algebra because otherwise, $A$ is isomorphic to $\mathrm{M}_2(F)$ and the Azumaya algebra $\mathrm{M}_2(O)$ would extend $A$. Statements (i) and (ii) follow from \cite[\S 12.8, \S 13.2]{Reiner}.

The ramification degree of $D$ is defined to be the integer $e \ge 1$ such that $J^e = D \pi$. It is shown in \cite[\S 13.3, \S 14.3]{Reiner} that $D/J$ is a division algebra of dimension $f$ over $K = O/O\pi$ for an integer $f$ such that $e f = 4$. If $e = 1$ then $J = D\pi$ and $D/D\pi$ is a central simple algebra. However, then $D$ is an Azumaya algebra over $O$ extending $A$, and we supposed that no such Azumaya algebra exists. Thus $e = 2$ or $e =4$.

If $e = 4$ then $J^4 = D \lambda^4 = D \pi$. Then $\lambda^4/\pi$ would be a unit of $D$, implying that $n(\lambda)^4/n(\pi) = n(\lambda)^4/\pi^2$ is a unit of $O$, which is impossible because $n(\lambda) \in O$ and $\pi$ is a uniformizer in $O$. Therefore $e = f = 2$, which proves (iii).

Finally suppose $z$ and $\tilde{z}$ are as in (iv). Then $F(z) \subset A$ must be a quadratic extension of $F$. If $k(\tilde{z})/k$ is not separable, the characteristic of $k$ must be $2$ and $\tilde{z}^2 \in k$. However, then
\[
n(z) = \mathrm{Norm}_{F(z)/F}(z) = 1
\]
has image
\[
1 = \mathrm{Norm}_{k(\tilde{z})/k} (\tilde{z}) = \tilde{z}^2
\]
in $D/J = k(\tilde{z})$. Since $k$ has characteristic $2$, we get $\tilde{z} = 1$, contradicting the assumption that $k(\tilde{z})/k$ has degree $2$. Thus $k(\tilde{z})/k$ is separable.

Since $F$ has characteristic $0$, we can find some $d$ with $d^2 = \alpha \in F$ and $F(z) = F(d)$. Then $A$ is a two dimensional left vector space over $F(d)$ and the conjugation action of $d$ on $A$ defines a non-trivial $F(d)$ linear automorphism of order $2$. It follows from splitting $A$ into the $\pm 1$ eigenspaces for this automorphism that
\[
A \cong F(z) \oplus F(z)w
\]
for some non-zero $w \in A$ with $dwd^{-1} = -w$, hence conjugation by $d$ carries the quadratic extension $F(w)$ of $F$ to itself. Since $-w$ is a conjugate of $w$ over $F$, we see that $w^2 \in F$, and in fact $\{d^2 ,w^2 \}$ is a Hilbert symbol for $A$ over $F$.

Write $z = a + bd$ for some $a, b \in F$ with $b \ne 0$. Then
\[
w z w^{-1} = a - bd
\]
is the conjugate of $z$ over $k$ in $k(z)$. Since $z$ has norm $1$, we must have $w z w^{-1} = z^{-1}$. Thus
\[
z w z^{-1} = zw wzw^{-1} = z^2 w,
\]
since $w^2 \in F$. Therefore, the characteristic polynomial in $F(z)[t]$ for the conjugation action of $z$ on $A$ as a two-dimensional left $F(z)$-vector space, is $(t - 1) (t - z^2)$. The extension $F(z)$ is quadratic and unramified over $F$ with ring of integers $O' = O[z]$ since $k(\tilde{z})$ is a separable quadratic extension of $k$.

Then $D$ is a rank two $O'$-module inside the two-dimensional $F(z)$-vector space $A$, and $D$ is preserved by conjugation by $z$. It follows that $(t-1)(t-z^2)$ is also the characteristic polynomial for the left $O'$-linear automorphism of $D$ given by conjugation by $z$. Since $D$ is free of rank two over $O'$, we see that the characteristic polynomial for the conjugation action of $z$ on $D/D\pi$ is $(t-1)(t-\tilde{z}^2)$. This action preserves the one-dimensional $k(z)$-subspace $J/D\pi$ of $D/D\pi$ and induces the trivial action on $D/J = k(\tilde{z})$. Therefore we conclude that conjugation by $z$ must induce left multiplication by $\tilde{z}^2$ on $J/D\pi = J/J^2$. This proves (iv).
\end{proof}
%%%%%%%%%%%%%%%%%%%%

%%%%%%%%%%%%%%%%%%%%
\subsection{Tate--Shafarevich groups}\label{TSgroups}
%%%%%%%%%%%%%%%%%%%%

In this subsection we assume the notation of \S \ref{s:ramification}. In particular, we assume that $C$ is a regular projective curve. One should keep in mind that the canonical component in Theorem \ref{non-trivial} is affine and generally singular, so the results of this section apply to the non-singular projective model of the canonical component rather than to the canonical component itself.

Let $O$ be the ring of $S$-integers $O_{{k},S}$ of a number field ${k}$ for some finite set $S$ of finite places of ${k}$. We begin by recalling some results of Stuhler \cite{Stuhler} and Demeyer--Knus \cite{DemeyerKnus} concerning Brauer groups and Tate--Shafarevich groups.

Let $\overline{{k}}$ be an algebraic closure of ${k}$ and let $J(C)$ be the Jacobian of $C$. The group $J(C)(\overline{{k}})$ is, by definition, $\mathrm{Pic}^0(C)(\overline{{k}})$. Let $V({k})$ be the set of all places of ${k}$. For $v \in V({k})$, let $\overline{{k}}_v$ be an algebraic closure of the completion ${k}_v$ that contains $\overline{{k}}$. We can identify $\mathrm{Gal}(\overline{{k}}_v/{k}_v)$ with a decomposition subgroup of $\mathrm{Gal}(\overline{{k}}/{k})$, and there is a restriction map
\[
r_v: H^1(\mathrm{Gal}(\overline{{k}}/{k}),\mathrm{Pic}^0(C)(\overline{{k}})) \to H^1(\mathrm{Gal}(\overline{{k}}_v/{k}_v),\mathrm{Pic}^0(C)(\overline{{k}}_v)).
\]
See \cite[Ch.\ III]{Mil}.

In \cite[Def.\ 1]{Stuhler}, the Tate--Shafarevich group of $C$ relative to $O_{{k},S}$ is defined to be
\[
\Sha({k},O_{{k},S},\mathrm{Pic}^0(C)) = \bigcap_{v \in V_f(S)} \ \mathrm{Ker}(r_v),
\]
where $V_f(S)$ is the set of finite places of ${k}$ not in $S$. The usual definition of the Tate--Shafarevich group of $\mathrm{Pic}^0(C)$ is
\begin{equation}\label{eq:traditional}
\Sha({k},\mathrm{Pic}^0(C)) = \bigcap_{v \in V({k})} \ \mathrm{Ker}(r_v).
\end{equation}
These two definitions are thus related by
\begin{equation}\label{eq:relationSha}
\Sha({k},\mathrm{Pic}^0(C)) = \bigcap_{v \in V_{real}({k}) \cup S} \ \mathrm{Ker} ( r_v |_{\Sha({k},O_{{k},S},\mathrm{Pic}^0(C))}),
\end{equation}
where $V_{real}({k})$ is the set of real places of ${k}$. The following result follows from the proofs of \cite[Thm.\ 1, Thm.\ 2, Thm.\ 3]{Stuhler}:

%%%%%%%%%%%%%%%%%%%%
\begin{thm}[Stuhler]\label{thm:TStheorem}
There is a complex
\begin{equation}
\label{Stuhler}
\mathrm{Br}(O_{{k},S}) \to \mathrm{Br}(\mathcal{C}) \to \Sha({k},O_{{k},S},\mathrm{Pic}^0(C))
\end{equation}
in which $\mathrm{Br}(O_{{k},S}) \to \mathrm{Br}(\mathcal{C}) $ is $f^*$ for $f:\mathcal{C} \to \mathrm{Spec}(O_{{k},S})$ the structure morphism. This complex is a short exact sequence if there is a section $s: \mathrm{Spec}(O_{{k},S}) \to \mathcal{C}$ of $f$. Since $\mathcal{C}$ is projective, such a section exists if and only if $C$ has a point defined over ${k}$.
\end{thm}
%%%%%%%%%%%%%%%%%%%%

%%%%%%%%%%%%%%%%%%%%
\begin{cor}\label{cor:easyStuhler}
Suppose $s$ is a section of $f$ in Theorem \ref{thm:TStheorem}. If $\mathrm{Br}^0_s(\mathcal{C})$ is the kernel of $s^*:\mathrm{Br}(\mathcal{C}) \to \mathrm{Br}(O_{{k},S})$, then $\mathrm{Br}^0_s(\mathcal{C})$ is isomorphic to $\Sha({k},O_{{k},S},\mathrm{Pic}^0(C))$.
\end{cor}
%%%%%%%%%%%%%%%%%%%%

We also need the following result, which was proven by Demeyer and Knus in \cite[p. 228-229]{DemeyerKnus}. They remark that this result goes back to work of E.\ Witt.

%%%%%%%%%%%%%%%%%%%%
\begin{thm}[Demeyer--Knus, Witt]\label{thm:realBrauer}
Suppose that $Y$ is a complete non-singular irreducible curve over the real numbers $\mathbb{R}$ and let $\{Y_i\}_{i = 1}^m$ be the set of connected components of the topological space $Y(\mathbb{R})$ of real points of $Y$.
\begin{enumerate}

\item Each $Y_i$ is topologically isomorphic to a real circle.

\item For each $Y_i$ pick a point $y_i \in Y_i$. If $A$ is an Azumaya algebra on $Y$, then $A \otimes_{Y} k(y_i)$ is an Azumaya algebra over $\mathbb{R}$.

\item Assuming (2), let $c_i(A)$ be the class of $A \otimes_{Y} k(y_i)$ in the Brauer group
\[
\Br(k(y_i)) \cong \Br(\mathbb{R}) \cong \mathbb{Z}/2.
\]
This does not depend on the choice of the point $y_i$ on $Y_i$ and the map
\[
\Br(Y) \to \prod_{i = 1}^m \Br(Y_i) \cong (\mathbb{Z}/2)^m
\]
sending the class of $A$ in $\Br(Y)$ to $(c_i(A))_{i = 1}^m$ is an isomorphism.

\end{enumerate}
\end{thm}
%%%%%%%%%%%%%%%%%%%%

The following records what we will use to prove our results.

%%%%%%%%%%%%%%%%%%%%
\begin{cor}\label{cor:morecor}
Suppose $S = \emptyset$ in Corollary \ref{cor:easyStuhler}, i.e., $O_{{k},S} = O_{k}$. Let $[\mathcal{A}]$ be an element of $\mathrm{Br}(\mathcal{C})$ represented by a quaternion Azumaya algebra over $\mathcal{C}$. Then:
\begin{enumerate}
\item The class $[\mathcal{A}]$ has order $1$ or $2$.

\item The image of $[\mathcal{A}]$ in $\Sha({k},O_{{k},S},\mathrm{Pic}^0(C))$ defines an element in the subgroup $\Sha({k},\mathrm{Pic}^0(C))$ if and only if for every real place $v \in V_{real}({k})$, the pullback $\mathcal{A}_v$ of $\mathcal{A}$ to an Azumaya algebra on $C \otimes_{k} {k}_v$ is trivial.

\item The conclusion of (2) holds if and only if for every point $y$ of $\mathcal{C}(\mathbb{R})$ the restriction of $\mathcal{A}$ to $y$ defines the matrix algebra $\mathrm{M}_2(k(y)) \cong \mathrm{M}_2(\mathbb{R})$ rather than the real quaternions $\mathbb{H}_\R$ over $k(y) \cong \mathbb{R}$.

\end{enumerate}
\end{cor}
%%%%%%%%%%%%%%%%%%%%

%%%%%%%%%%%%%%%%%%%%
\subsection{Proof of Theorem \ref{thm:MainKnotIntegral}}
%%%%%%%%%%%%%%%%%%%%

As before, let $M = S^3 \ssm K$ be a hyperbolic knot $K$ in $S^3$, and let 
$\Gam = \pi_1(M)$. Let $\mathfrak{C}\subset X(K)$ be the canonical curve of $M$ over $\C$, and let $k$ be the associated field of constants. Let $C = C_M$ be a canonical component in $X(K)_k$, with $\wt{C}$ be the the normalization of a projective closure of $C$. Then $\wt{C}$ is a smooth geometrically irreducible curve over ${k}$.

Let $\mathcal{C}_S$ be a regular projective integral model of $\wt{C}$ over the ring $O_{{k},S}$ of $S$-integers of ${k}$ for some finite set of finite places $S$ of ${k}$. We have a canonical quaternion algebra $A_{k(C)}$ over the function field
\[
k(C) \cong k(\wt{C}) \cong k(\mathcal{C}_S).
\]
This algebra was constructed as the $k(C)$ subalgebra of $\mathrm{M}_2(F)$ generated by the image of a representation $\rho:\Gamma \to \mathrm{SL}_2(F)$ whose character $\chi_{\rho}$ defines the generic point of $C$, where $F$ is a sufficiently large finite extension of $k(C)$.

We now need the following lemma, which should be compared with Proposition \ref{prop:IdealExtension}(2). Indeed the proof proceeds exactly as before until the last step.

%%%%%%%%%%%%%%%%%%%%
\begin{lem}\label{lem:notinfinite}
Suppose $P$ is a codimension one point of $\mathcal{C}_S$ such that $A_{k(C)}$ does not extend over $P$. Then $P$ is not a point at infinity in the sense that the trace $\chi_{\rho}(\gamma)$ lies in the local ring $O_P = O_{\mathcal{C}_S,P}$ of $P$ on $\mathcal{C}_S$ for all group elements $\gamma \in \Gamma$.
\end{lem}
%%%%%%%%%%%%%%%%%%%%

%%%%%%%%%%%%%%%%%%%%
\begin{proof}
Recall that the algebra $A_{k(C)}$ must be a division algebra over $k(C)$, since otherwise it would trivially extend over $P$. As in the proof Proposition \ref{prop:IdealExtension}(2), suppose for a contradiction that we can find $\gamma\in\Gam$ with $\chi_{\rho}(\gamma) \not \in O_P$ and a Hilbert symbol for $A_{k(C)}$ over $k(C)$ of the form $\{\alpha',\beta\}$ with $\alpha' = \chi_{\rho}(\gamma)^2 - 4$ and $\beta \in k(C)$. Since $\chi_{\rho}(\gamma) \not \in O_P$, $\chi_{\rho}(\gamma) \ne 0$.

As before, we can multiply $\alpha'$ by $\chi_{\rho}(\gamma)^{-2}$ to give another Hilbert symbol $\{\alpha,\beta\}$ for $A_{k(C)}$ with
\[
\alpha = 1 - \frac{4}{\chi_{\rho}(\gamma)^2}.
\]
We now finish the proof by noting that
\[
\mathrm{ord}_P \left(\frac{4}{\chi_{\rho}(\gamma)^2}\right) = 2\, \mathrm{ord}_P(2) - 2\, \mathrm{ord}_P(\chi_{\rho}(\gamma)) > 2\, \mathrm{ord}_P(2).
\]
Indeed, Proposition \ref{prop:Parimala-Suresh} implies that $A_{k(C)}$ extends over $P$, contrary to our hypothesis.
\end{proof}
%%%%%%%%%%%%%%%%%%%%

Now we can prove the main technical theorem that connects the set $S$ in Theorems \ref{thm:MainDehn} and \ref{thm:MainKnotIntegral} to the reduction of the Alexander polynomial modulo rational primes.

%%%%%%%%%%%%%%%%%%%%
\begin{thm}\label{thm:filter}
Suppose that $P$ is a codimension one point of $\mathcal{C}_S$ over which $A_{k(C)}$ does not extend. Let $k(P)$ be the residue field of the local ring $O_P$. Then there is an element $\tilde{z}$ of an algebraic closure $\overline{k(P)}$ of $k(P)$ with the following properties:
\begin{enumerate}

\item[i.] The extension $k(P)(\tilde{z})$ is separable and quadratic over $k(P)$ and $\tilde{z}$ has norm $1$ to $k(P)$.

\item[ii.] Considering the Alexander polynomial $\Delta_K(t)$ of $K$ as an element of $\overline{k(P)}[t, t^{-1}]$, we have that $\Delta_K(\tilde{z}^2)=0$.

\end{enumerate}
\end{thm}
%%%%%%%%%%%%%%%%%%%%

%%%%%%%%%%%%%%%%%%%%
\begin{proof}
Lemma \ref{lem:notinfinite} implies that $P$ is not a point at infinity, in the sense that $\chi_{\rho}(\gamma)$ lies in $O_P$ for all $\gamma \in \Gamma$. Non-degeneracy of the quadratic form associated with the trace $\tr:A_{k(C)} \to k(C)$ now shows that the $O_P$ subalgebra of $\mathrm{M}_2(F)$ generated by the image of $\rho:\Gamma \to \mathrm{SL}_2(F)$ is contained in a finitely generated $O_P$ submodule of $A_{k(C)}$. Therefore this subalgebra is an $O_P$ order $D_0$ in $A_{k(C)}$ that has rank $4$ as a free $O_P$-module.

Let $\hat{O}_P$ be the completion of $O_P$ and $\hat{F}$ be the fraction field of $\hat{O}_P$. Then $\hat{A} = A_{k(C)} \otimes_{O_P} \hat{F}$ is a quaternion algebra over $\hat{F}$. This quaternion algebra cannot be extended to an Azumaya algebra over $\hat{O}_P$, since otherwise $A_{k(C)}$ could be extended over $P$ by Lemma \ref{lem:parnice}. We can then pick a maximal $\hat{O}_P$-order $D$ in $\hat{A}$ containing $D_0$. Since $\rho$ has image in $\mathrm{SL}_2$, we conclude that $\rho$ gives an injective homomorphism $\rho:\Gamma \to D^1$ to the multiplicative group of units in $D$ with reduced norm $1$.

We now let $J$ be the unique maximal two-sided ideal of $D$ described in Proposition \ref{prop:structure}. For $s \ge 1$, let $U_s$ be the image of $D^1$ in $D/ J^s$ and $W_s$ be the image of $\rho(\Gamma)$ in $U_s$. Since $D/J$ is a quadratic field extension of $k(P)$, we know that $W_1$ is abelian. There is an exact sequence
\[
1 \to E_{s+1}\to U_{s+1} \to U_s \to 1,
\]
where $U_{s+1} \to U_s$ is reduction modulo $J^s$ and $E_{s+1}$ is the subgroup of elements of $(1 + J^s) / (1 + J^{s+1})$ sent to $1$ under the reduced norm.

Let $s$ be the largest integer such that the group $W_s$ is abelian. Since $\Gamma$ embeds into $D^1$ and $\Gamma$ is not abelian, we know that $s \ge 1$ and that $s$ must be finite. Note that $W_s$ is in fact cyclic, since $\Gamma$ has cyclic abelianization.

We now have an exact sequence
\[
1 \to (E_{s+1} \cap W_{s+1}) \to W_{s+1} \to W_s \to 1,
\]
where $W_{s+1}$ is not abelian and $W_s$ is cyclic. Then $E_{s+1} \cap W_{s+1} = Q_s$ is abelian, since it is a subgroup of $(1 + J^s)/(1 + J^{s+1})$. The action of $W_s$ on $(1 + J^s)/(1+ J^{s+1})$ by conjugation factors through the reduction map $W_s \to W_1$. Thus $W_s$ is cyclic and $W_{s+1}$ is not abelian, so we conclude that $W_1$ is not contained in $k(P)$ and the action of $W_1$ on $Q_s = E_{s+1} \cap W_{s+1}$ is not trivial. In particular, the action of $U_1 = (D/J)^1$ on $(1 + J^s)/(1 + J^{s+1})$ must be non-trivial.

We now choose a generator $g$ for $\Gamma$ modulo its commutator subgroup $\Gamma' = [\Gamma,\Gamma]$. Let $z = \rho(g) \in D^1$ and $\tilde{z}$ be the image of $z$ in $(D/J)^1$. Then $\tilde{z}$ generates $W_1$ and, since $W_1$ is not contained in $k(P)$, the quadratic extension $D/J$ of $k(P)$ must be generated over $k(P)$ by $\tilde{z}$. Proposition \ref{prop:structure} then implies that $k(P)(\tilde{z})$ is a separable quadratic extension of $k(P)$.

For any $s \ge 1$ we have an isomorphism of groups
\[
(1 + J^s)/(1 + J^{s+1}) = J^s/J^{s+1}
\]
respecting the conjugation action of $z$. Then $J^2 = D\pi_P$, where $\pi_P$ is a uniformizer in $O_P$ and $\pi_P$ commutes with $z$. Since $W_1$ acts non-trivially on $Q_s \subset (1 + J^s)/(1 + J^{s+1}) = J^s/J^{s+1}$, we conclude from Proposition \ref{prop:structure}(iv) that $s$ must be odd. Furthermore the action of $\tilde{z} \in W_1$ corresponds to conjugation by $z$, which in turn corresponds to left multiplication by $\tilde{z}^2$ on the one-dimensional $k(P)(\tilde{z})$-vector space $J^s/J^{s+1}$.
 
Choose a non-trivial element
\[
h \in Q_s \subset (1 + J^s)/(1 + J^{s+1}) = J^s/J^{s+1}.
\]
We can define a map $\nu:Q_s \to k(P)(\tilde{z}) $ by $h' = \nu(h')\cdot h$ with respect to the structure of $J^s/J^{s+1}$ as a one-dimensional vector space over $k(P)(\tilde{z})$. Then $\nu$ is a group homomorphism. The commutator subgroup
\[
\Gamma' = [\Gamma,\Gamma]
\]
has trivial image in $W_s$, since $W_s$ is an abelian quotient of $\Gamma$. Therefore, the homomorphism $\Gamma \to W_{s+1}$ sends $\Gamma'$ to $Q_s$.

Restricting $\nu$ to the image of $\Gamma'$ gives a homomorphism $r: \Gamma' \to k(P)(\tilde{z})$. Here $r$ must be non-trivial, since $W_{s+1}$ is not abelian. We have shown that under $r$, the conjugation action of $z$ on $\Gamma'$ corresponds to left multiplication by $\tilde{z}^2$ on $k(P)(\tilde{z})$. Since $k(P)(\tilde{z})$ is abelian, $r$ factors through a non-trivial homomorphism $\overline{r}:V \to k(P)(\tilde{z})$, where $V = \Gamma'/[\Gamma',\Gamma']$.

Now recall that $V$ is a finitely generated torsion module for the group ring
\[
\mathbb{Z}[\Gamma/\Gamma'] = \mathbb{Z}[t,t^{-1}].
\]
The Alexander polynomial $\Delta_K(t) \in \mathbb{Z}[t,t^{-1}]$ is a generator of the $0^{th}$ Fitting ideal of $V$. General properties of Fitting ideals (e.g., see \cite[p.\ 671]{DummitFoote}) imply that the image of $\Delta_K(t)$ in $k(P)(\tilde{z})[t,t^{-1}]$ is a generator for the $0^{th}$ Fitting ideal of $V \otimes_{\mathbb{Z}} k(P)(\tilde{z})$ as a module for $k(P)(\tilde{z})[t,t^{-1}]$. Since $V \otimes_{\mathbb{Z}} k(P)(\tilde{z})$ is a finitely generated torsion module for the pid $k(P)(\tilde{z})[t,t^{-1}]$, we conclude $V$ is finite dimensional over $k(P)(\tilde{z})$ and the $0^{th}$ Fitting ideal is generated by the characteristic polynomial associated with the action of the generator $g$ of $\Gamma/\Gamma'$.

Here $z = \rho(g) \in D^1$ shows that the action of $g$ comes from conjugation by $z$. We showed that conjugation by $z$ on $V$ corresponds to left multiplication by $\tilde{z}^2$ on $k(P)(\tilde{z})$ under the non-trivial homomorphism
\[
\overline{r}: V = \Gamma'/[\Gamma',\Gamma'] \to k(P)(\tilde{z}).
\]
It follows that $\tilde{z}^2$ must be a root of the characteristic polynomial for the action of $g$ on $V$, so $\tilde{z}^2$ is a root in $\overline{k(P)}$ of $\Delta_K(t)$. This completes the proof of the theorem.
\end{proof}
%%%%%%%%%%%%%%%%%%%%

%%%%%%%%%%%%%%%%%%%%
\begin{cor}
\label{cor:almostdone}
With the hypotheses and notation of Theorem \ref{thm:filter}, let $\mathbb{F}$ be the prime subfield of $k(P)$, so either $\mathbb{F} = \mathbb{Q}$ or $\mathbb{F} = \mathbb{F}_\ell$ for some prime $\ell$, and let $\overline{\mathbb{F}}$ be an algebraic closure of $\mathbb{F}$. Then there is an element $u$ of $\overline{\mathbb{F}}$ such that $u^2$ is a root of $\Delta_K(t)$ in $\overline{\mathbb{F}}$ and $\mathbb{F}(u)$ is a separable quadratic extension of $\mathbb{F}(u + u^{-1})$.
\end{cor}
%%%%%%%%%%%%%%%%%%%%

%%%%%%%%%%%%%%%%%%%%
\begin{proof}
Consider the element $\tilde{z}$ of $\overline{k(P)}$. We showed in Theorem \ref{thm:filter} that $\tilde{z}^2$ is a root of $\Delta_K(t) \in \mathbb{Z}[t,t^{-1}]$ in $\overline{k(P)}$. Since $\Delta_K(t)$ has coefficients in $\mathbb{Z}$, this implies that $\tilde{z}^2$ is algebraic over $\mathbb{F}$, so $u = \tilde{z}$ lies in $\overline{\mathbb{F}}$. Since $\tilde{z}$ has norm $1$ to $k(P)$ and $\tilde{z}$ is quadratic and separable over $k(P)$, we know that $\tilde{z}^{-1}$ is the other conjugate of $\tilde{z}$ over $k(P)$.

Therefore
\[
u + u^{-1} = \tilde{z} + \tilde{z}^{-1} \in k(P),
\]
and so $\mathbb{F}(u)$ is at most quadratic over $\mathbb{F}(u + u^{-1})$. If $\mathbb{F}(u) = \mathbb{F}(u + u^{-1})$, then
\[
k(P)(u) = k(P)(u+u^{-1})= k(P),
\]
which contradicts the fact that $k(P)(u) = k(P)(\tilde{z})$ is quadratic over $k(P)$. Therefore $\mathbb{F}(u)/ \mathbb{F}(u + u^{-1})$ is quadratic, and this extension is separable because $\mathbb{F}$ is a prime field.
\end{proof}
%%%%%%%%%%%%%%%%%%%%

We are now prepared for:

%%%%%%%%%%%%%%%%%%%%
\begin{proof}[Completion of the proof of Theorem \ref{thm:MainKnotIntegral}]
Let $S$ be a finite set of rational primes with the properties stated in the theorem, and suppose $\ell$ is a prime not in $S$. Corollary \ref{cor:almostdone} implies that $A_{k(C)}$ extends over every codimension one point $P$ of $\mathcal{C}_S$ that lies in the fiber of $\mathcal{C}_S$ over $\ell$. With the assumptions of Theorem \ref{thm:MainKnotIntegral}, $A_{k(C)}$ extends over every codimension one point on the general fiber $\wt{C}_M$ of $\mathcal{C}_S$, which implies that $A_{k(C)}$ extends over all of $\mathcal{C}_S$.
\end{proof}
%%%%%%%%%%%%%%%%%%%%

%%%%%%%%%%%%%%%%%%%%
\begin{rems}\label{rem:cebot}{\ }
\begin{enumerate}

\item Since $P$ in Corollary \ref{cor:almostdone} can have characteristic $0$, the same argument gives a different proof of the criterion in Theorem \ref{thm:MainKnot} for the Azumaya algebra $A_{k(C)}$ over the function field $k(C) = k(\mathcal{C}_S)$ to extend to the general fiber $\wt{C}$ of $\mathcal{C}_S$. Recall that the proof of Theorem \ref{thm:MainKnot} used the tame symbol, which is not available in characteristic two, so that argument does not suffice to prove the results in this section. However, we make use of the tame symbol for several other consequences of Theorem \ref{thm:MainKnot}, hence we make non-trivial use of each of the two arguments.

\item Conversely, suppose that condition $(\star)$ in Theorem \ref{thm:MainKnot} holds. Let $S_0$ be a sufficiently large set of rational primes so that the leading coefficient of $\Delta_K(t) \in \mathbb{Z}[t,t^{-1}]$ is a unit outside of $S_0$. Then the roots of $\Delta_K(t)$ in $\overline{\mathbb{Q}}$ are integral outside of $S_0$. Furthermore, if $\ell$ is a prime not in $S_0$ and $\overline{w}$ is a root of $\Delta(t)$ in $\overline{(\mathbb{Z}/\ell)}$ then $\overline{w}$ is the reduction modulo a prime over $\ell$ of a root $w$ of $\Delta_K(t)$ in the ring of all algebraic numbers integral outside of $S_0$. The hypothesis that $\mathbb{Q}(w) = \mathbb{Q}(w + w^{-1})$ implies that, by possibly making a finite enlargement of $S_0$, we can assume that each such $w$ is an $S_0$-integral combination of powers of $w + w^{-1}$. This forces there to be a finite set of primes $S$ with the properties stated in Theorem \ref{thm:MainKnotIntegral}.

\end{enumerate}
\end{rems}
%%%%%%%%%%%%%%%%%%%%

%%%%%%%%%%%%%%%%%%%%
\subsection{Proof of Theorem \ref{non-trivial}}
%%%%%%%%%%%%%%%%%%%%

We assume the notation and hypotheses from the statement of the theorem. Since the Alexander polynomial $\Delta_K(t)$ is assumed to be $1$, we can let $S$ be the empty set in Theorem \ref{thm:MainKnotIntegral}. Then Theorem \ref{thm:MainKnotIntegral} shows that there is an extension $\mathcal{A}$ of $A_{k(\wt{C})}$ over all of $\mathcal{C}$ for any regular projective model $\mathcal{C}$ of $\wt{C}$ over the ring of integers $O_{k}$ of ${k}$. 

Statement (1) of Theorem \ref{non-trivial} is now a consequence of Theorem \ref{thm:TStheorem}. In order for the class $\beta([\mathcal{A}])$ to lie in $\Sha({k},\mathrm{Pic}^0(\tilde{C}))$, it is necessary and sufficient by Corollary \ref{cor:morecor} that the restriction
\[
\mathcal{A}_y = \mathcal{A} \otimes_{O_\mathcal{C}} k(y)
\]
of $\mathcal{A}$ to $y$ be isomorphic to $\M_2(\mathbb{R})$ rather than the real quaternions $\mathbb{H}_{\mathbb{R}}$ for every real point $y \in \wt{C}(\mathbb{R})$. Corollary \ref{cor:morecor} also shows that $\wt{C}(\mathbb{R})$ is a finite (possibly empty) union of real circles, and that the isomorphism type of $\mathcal{A}_y$ is constant as $y$ varies over each of these circles.

Since $\wt{C}$ is the normalization of a projective closure of $C$, there is a finite closed subset $T \subset \wt{C}$ such that $\wt{C} \ssm T = C \ssm C_{sing}$ is the complement of the (finite) singular locus $C_{sing}$ of $C$. We conclude that if $\mathcal{A}_y$ is isomorphic to $\mathbb{H}_{\mathbb{R}}$ for any point $y \in \wt{C}(\mathbb{R})$, then this is true for a non-empty union of real circles of such $y$ as well as for a subset $T$ of $C(\mathbb{R})$ which is the complement of a finite set inside a non-empty union of real circles. Since the multiplicative group $\mathbb{H}_{\mathbb{R}}^1$ of quaternions of reduced norm $1$ is isomorphic to $\SU(2)$, we find in this case that the points of $T$ correspond to characters of $\SU(2)$ representations of our knot group.

On the other hand, suppose there is no $y \in \wt{C}(\mathbb{R})$ such that $\mathcal{A}_y$ is isomorphic to $\mathbb{H}_{\mathbb{R}}$ and $y' \in C(\mathbb{R})$ corresponds to an $\SU(2)$ representation. Regarding $\SU(2)$ as $\mathbb{H}_{\mathbb{R}}^1$, we see that the $\mathbb{R}$-algebra $A_\rho$ generated by any representation $\rho$ with character $y'$ is a subalgebra of $\mathbb{H}_{\mathbb{R}}$. If $A_\rho$ is not $\mathbb{H}_{\mathbb{R}}$ then $\rho$ must be reducible. However, Proposition \ref{CCGLS} shows that we can take $\rho$ to be a reducible non-abelian representation associated with a zero of the Alexander polynomial of $K$. Since we assumed that the Alexander polynomial is trivial, there are no such zeros, hence $A_\rho = \mathbb{H}_{\mathbb{R}}$ and $y'$ must lie in $C_{sing}(\R)$ since we are supposing no such point exists in the smooth locus. This completes the proof of part (2) of Theorem \ref{non-trivial}.

Finally, the hypotheses of part (3) of Theorem \ref{non-trivial} are that $\tilde{C}$ has no real points but that there is a point $P$ of $\tilde{C}$ defined over the field of constants ${k}$ of $C$. The Zariski closure of $P$ gives a section of $\mathcal{C} \to \mathrm{Spec}(O_{k})$, so Theorem \ref{thm:TStheorem} implies that we have an exact sequence
\begin{equation}\label{eq:exactBr}
\mathrm{Br}(O_{{k}}) \to \mathrm{Br}(\mathcal{C}) \to \Sha({k},O_{{k}},\mathrm{Pic}^0(C)).
\end{equation}
Suppose that $[\mathcal{A}] \in \mathrm{Br}(\mathcal{C})$ has trivial image in $\Sha({k},O_{{k}},\mathrm{Pic}^0(C))$.

Exactness of \eqref{eq:exactBr} now shows that $[\mathcal{A}]$ is the pull-back to $\mathcal{C}$ of a class $\sigma$ in $\Br(O_{k})$. Such a $\sigma$ can only be non-trivial at the real places of ${k}$. If $\sigma$ is non-trivial at some real place, then, since we assumed that $\wt{C}$ has a point over ${k}$, there will be a point $y$ of $\mathcal{C}(\mathbb{R}) $ where $\mathcal{A}_y$ is isomorphic to $\mathbb{H}_{\mathbb{R}}$. This implies that there would be a curve of $\SU(2)$ characters on $C$ by the above arguments, contrary to the hypothesis of part (3) of Theorem \ref{non-trivial}. Therefore $[\mathcal{A}] = 0 $ in $\mathrm{Br}(\mathcal{C})$ as claimed in part (3). This completes the proof of the non-trivial assertion in part (3) of Theorem \ref{thm:MainKnotIntegral}.\qed

\medskip

%%%%%%%%%%%%%%%%%%%%
\begin{rem}
Theorem \ref{thm:realBrauer} shows that
every real circle on $\wt{C}(\R)$ containing the character of an irreducible $\SU(2)$ representation contains only characters of $\SU(2)$ representations. 
In particular, this is the case via Theorem \ref{thm:realBrauer} whenever there is the character of an irreducible $\SU(2)$ representation on the canonical component $C$ and it is known that $A_{k(C)}$ extends to an Azumaya algebra over $\wt{C}$. An instance of this is when (2) fails in Theorem \ref{non-trivial}. However, for the character variety of an arbitrary knot group, it is not always the case that real arcs consist only of characters of representations into a fixed real algebraic subgroup of $\SL_2(\C)$.

\medskip

For example, in \cite[Cor.\ 1.4(ii)]{HP} the authors show that if
$\lambda$ is the square root of a root of the Alexander polynomial
(i.e., $\lambda$ of \cite{HP} is our $w$)
and $|\lambda| = 1$, then there is a real arc parametrized by
$\{\chi_t\}$, $t \in (-\epsilon, \epsilon)$, with $\chi_t$ the
character of an irreducible $\SU(2)$ representation for $t > 0$ and
the character of an irreducible $\SU(1,1)\cong\SL(2,\R)$ representation for $t <
0$. If $\lambda = \pm 1$, then the representation is parabolic and
hence cannot lie on the canonical component by Corollary
\ref{nounipotent}. In particular, if $C$ contains a
non-abelian reducible representation of the kind described in the
paragraph above, then Theorem \ref{thm:realBrauer}
shows that $A_{k(C)}$ cannot extend to an Azumaya
algebra over $\wt{C}$. 
Note also that if $\lambda \notin \R$, then $\lambda + 1 /
\lambda$ is real, and hence condition $(\star)$ fails. \end{rem}
%%%%%%%%%%%%%%%%%%%%

%%%%%%%%%%%%%%%%%%%%
\section{Examples}\label{sec:Examples}
%%%%%%%%%%%%%%%%%%%%

We begin with some general discussion about the examples to follow. For convenience, we introduce the following notation. Let $p(t)$ be a polynomial with integer coefficients. 

We say $p(t)$ is \emph{Azumaya positive} if condition $(\star)$ of Theorem \ref{thm:MainKnot} holds for any root $z$ of $p(t)$, and say that $p(t)$ is \emph{Azumaya negative} if condition $(\star)$ fails for some root $z$ of $p(t)$. Call a knot $K$ \emph{Azumaya positive} (resp.\ \emph{Azumaya negative}) if $\Delta_K(t)$ is Azumaya positive (resp.\ negative). Note that if the knot $K$ has trivial Alexander polynomial, it is certainly Azumaya positive. The challenge is to understand when knots are Azumaya positive or negative in the case when the Alexander polynomial is non-trivial (e.g., when the knot is fibered).

It is worth remarking that if the knot is Azumaya positive, then our construction produces an Azumaya algebra over the smooth projective model of the canonical component. However, if the knot is Azumaya negative, the Azumaya algebra may well extend over smooth projective model of the canonical component, since $\Delta_K(t)$ may be reducible and the canonical component may not contain any characters of reducible representations, or characters of reducible representations for which $(\star)$ does not hold.

Certifying that our construction \emph{cannot} extend to give an Azumaya algebra over the smooth projective model of the canonical component is more subtle, especially when the natural affine model has singular points. However, if the canonical component is the unique component containing characters of absolutely irreducible representations and the points associated with reducible representations are smooth points, then Theorem \ref{thm:BigConverse} implies that Azumaya negativity does indeed certify that our construction does not provide an Azumaya algebra over the smooth projective model.

In this section we will highlight a number of examples, most of which we summarize in the following theorem.

%%%%%%%%%%%%%%%%%%%%
\begin{thm}\label{example_conseq}{\ }
\begin{enumerate}

\item There are infinitely many fibered hyperbolic knots $K_n\subset S^3$ that are Azumaya positive.

\item There are infinitely many fibered hyperbolic knots $J_n\subset S^3$ that are Azumaya negative.

\item Let $T_m$ be the twist knot with $m\geq 1$ half twists. Then $T_m$ is Azumaya positive if and only if $m=2\ell$ is even and $\ell$ is either a square or the product of two consecutive integers.

\end{enumerate}
\end{thm}
%%%%%%%%%%%%%%%%%%%%

It is shown in \cite{MPV} that for a twist knot $T_m$ the canonical component $X_0$ is the unique component containing the character of an irreducible representation. When $m=2\ell$ is even and $\ell$ is either a square or the product of two consecutive integers, there is a finite set $S$ of places of $\Q$ such that, for all points $\chi \in X_0$ corresponding to characters of absolutely irreducible representations $\rho$, the quaternion algebra $A_\rho$ over $k_\rho$ is unramified outside of the finite places of $k_\rho$ over $S$. In particular, this applies to the points on $X_0$ associated with Dehn surgeries. Furthermore, since the Alexander polynomial of a twist knot is a quadratic polynomial without a double root, the points on $X_0$ associated with non-abelian reducible representations are smooth points by \cite{HP}. It follows from Theorem \ref{thm:BigConverse} that whether or not our Azumaya algebra extends over the smooth projective model is completely unambiguous.

%%%%%%%%%%%%%%%%%%%%
\subsection{Some polynomials that are Azumaya positive or negative}\label{sec:AZpolys}
%%%%%%%%%%%%%%%%%%%%

For convenience we describe some families of polynomials that are Azumaya positive or negative. All the polynomials will be the Alexander polynomial of some hyperbolic knot. Recall that if $K$ is a fibered knot it is known that $\Delta_K(t)$ is a monic reciprocal polynomial. Moreover, any monic reciprocal polynomial is the Alexander polynomial of a fibered hyperbolic knot; see \cite[Thm.\ 3.1]{St}.

A particularly interesting class of reciprocal polynomials are those arising as the irreducible polynomial of a \emph{Salem number}, i.e., a real algebraic integer $\lambda>1$ such that $1/\lambda$ is a Galois conjugate of $\lambda$ and all other Galois conjugates lie on the unit
circle. Denote the irreducible polynomial of a Salem number $\lambda$ by $p_\lambda(t)$. We will always assume that there is at least one Galois conjugate on the unit circle, so that the degree of $p_\lambda(t)$ is strictly greater than $2$.

We now prove the following two lemmas.

%%%%%%%%%%%%%%%%%%%%
\begin{lem}\label{poly_azumaya_negative}{\ }
\begin{enumerate}

\item Let $\lambda$ be a $n$-th root of unity for $n\geq 3$ and $\Phi_n(t)$
the $n$-th cyclotomic polynomial. Then $\Phi_n(t)$ is Azumaya negative.

\item Let $\lambda$ be a Salem number. Then $p_\lambda(t)$ and $p_\lambda(-t)$ are Azumaya negative.

\item Let $m\geq 1 $ be an odd integer and set
\[
q_m(t)= \frac{m+1}{2} t^2 - m t + \frac{m+1}{2}.
\]
Then $q_m(t)$ is irreducible, has both roots imaginary, and is Azumaya negative.
\end{enumerate}
\end{lem}
%%%%%%%%%%%%%%%%%%%%

%%%%%%%%%%%%%%%%%%%%
\begin{lem}\label{poly_azumaya_positive}{\ }
\begin{enumerate}

\item For any integer $a\geq 7$, the polynomial
\[
f_a(t)=t^4-at^3+(2a-1)t^2-at+1
\]
is irreducible with all roots real and positive. In addition, $f_a(t)$ is Azumaya positive when $a$ has the form $k^2+2$ for some $k\geq 3$.

\item Let $m=2\ell>0$ be an even integer, and set
\[
p_m(t)= \frac{m}{2} t^2 - (m + 1) t + \frac{m}{2}.
\]
Then $p_m(t)$ has both roots real and positive, and is Azumaya positive if and only if $\ell$ is either a square or the product of two consecutive integers.

\item The polynomial
\[
f(t)=t^8-3t^7 +5t^6-7t^5+9t^4-7t^3+5t^2-3t+1
\]
is irreducible, has all roots imaginary, and is Azumaya positive.

\end{enumerate}
\end{lem}
%%%%%%%%%%%%%%%%%%%%

%%%%%%%%%%%%%%%%%%%%
\begin{rem}
In the notation of Lemma \ref{poly_azumaya_positive}, the polynomials $f_a(t)$ can also be Azumaya positive for other values of $a$, e.g., when $a=7$.
\end{rem}
%%%%%%%%%%%%%%%%%%%%

%%%%%%%%%%%%%%%%%%%%
\begin{proof}[Proof of Lemma \ref{poly_azumaya_negative}]
To prove (1), since $\lambda$ is an $n$-th root of unity and $n\geq 3$, $\lambda$ is not real but $\lambda+1/\lambda$ is real. Moreover, if $\lambda$ is an $n$-th root of unity, $\sqrt{\lambda}$ is a $2n$-th root of unity. These remarks quickly lead to the proof of (1).

To prove (2) we begin with some preliminary comments. If $\lambda$ is a Salem number, then, since we are assuming that $\lambda$ has at least one non-real Galois conjugate, $\Q(\lambda)$ is not totally real. Note that the field $\Q(\lambda+1/\lambda)$ is totally real. In addition, it is known that $\lambda^n$ is a Salem number for any integer $n\geq 2$, and $\Q(\lambda)=\Q(\lambda^n)$.

First assume that $\lambda=u^{2n}$ for some Salem number $u\in \Q(\lambda)$. In this case
\[
\Q(\sqrt{\lambda})=\Q(\sqrt{u^{2n}})=\Q(u^n)=\Q(\lambda).
\]
However, this is a proper extension of the totally real field $\Q(u^n+1/u^n)$, so $p_\lambda(t)$ is Azumaya negative.

Now assume that $w=\sqrt{\lambda}\notin \Q(\lambda)$. Then $w$ satisfies a polynomial of degree $2$ over $\Q(\lambda)$, hence the degree of the irreducible polynomial of $w$ over $\Q$ is $2\deg(p_\lambda(t))$. The Galois conjugates of $w$ are $\pm w$, $\pm 1/w$ and the rest are non-real complex numbers on the unit circle. Note that $w$ is not a Salem number, as its minimal polynomial over $\Q$ has four distinct real roots. Nevertheless, the field $\Q(w+1/w)$ is still totally real, and so again different from $\Q(w)$ as required.

Now consider $p_\lambda(-t)$. This has two real negative roots $-\lambda$ and $-1/\lambda$, and all other roots still lie on the unit circle. In this case one readily sees that if $w=\sqrt{-\lambda}$, the field $\Q(w)$ is totally imaginary, and the field $\Q(w+1/w)$ has real embeddings arising from the roots on the unit circle. Therefore $\Q(w)\neq\Q(w+1/w)$, which completes (2).

\medskip

For (3), suppose $m = 2 \ell + 1$ is odd. Then the roots of $q_m(t)$ are
\[
z = \frac{(2 \ell + 1) \pm i \sqrt{4 \ell + 3}}{2(\ell + 1)}.
\]
In particular, $\Q(z) / \Q$ is always imaginary quadratic. This proves the first two claims in (3). Note that $4 \ell + 3$ is never a square, so $\Q(z)$ is never $\Q(i)$.

We also have that
\[
z + z^{-1} = \frac{2 \ell + 1}{\ell + 1} \in \Q.
\]
Furthermore, notice that
\[
(w + w^{-1})^2 - 2 = w^2 + w^{-2} = z + z^{-1},
\]
so $\Q(w + w^{-1}) / \Q(z + z^{-1}) = \Q$ is either degree one or two. Then $w$ is a root of the equation
\[
p(x) = x^2 - (w + w^{-1})x + 1 \in \Q(w + w^{-1})[x],
\]
so $\Q(w) / \Q(w + w^{-1})$ is degree one or two.

Since $\Q(z) / \Q$ is quadratic and $\Q(w + w^{-1})$ is at most quadratic, it follows that $\Q(w) = \Q(w + w^{-1})$ if and only if
\[
\Q(w) = \Q(w + w^{-1}) = \Q(z).
\]
If $\Q(w) = \Q(z)$, then
\[
\frac{(2 \ell +1) \pm i \sqrt{4 \ell + 3}}{2(\ell + 1)} = (a + b i \sqrt{4 \ell + 3})^2
\]
for some $a, b \in \Q$, which gives:
\begin{align*}
2 a b &= \pm \frac{1}{2(\ell + 1)} \\
a^2 - (4 \ell + 3) b^2 &= \frac{2 \ell + 1}{2(\ell + 1)}
\end{align*}
These two equations combine to give
\[
a^2 - \frac{4 \ell + 3}{16 a^2 (\ell + 1)^2} = \frac{2 \ell + 1}{2 (\ell + 1)},
\]
which implies that
\[
a^2 \in \left\{ \frac{4 \ell + 3}{4(\ell + 1)}, - \frac{1}{4(\ell + 1)} \right\}.
\]
However, $(4 \ell + 3)$ and $4 \ell + 4$ are coprime, so $(4 \ell + 3) / (4 \ell + 4)$ is the square of a rational number if and only if $4 \ell + 3$ and $4 \ell + 4$ are both squares, but $4 \ell + 3$ is never a square. Also, $-1/(4 \ell + 4)$ is clearly not the square of a rational number. This proves that $\Q(w) / \Q(z)$ must be quadratic, and so this completes the proof that $q_m(t)$ is Azumaya negative.
\end{proof}
%%%%%%%%%%%%%%%%%%%%

\begin{rem}
\label{unitcircle}
Note that the arguments used in the proofs of (1) and (2) above shows the following. If $p(t)$ is any polynomial with integer coefficients that has a root $\lambda$ lying on the unit circle and $\lambda\neq \pm 1$, then condition $(\star)$ fails for $\lambda$. For if $w=\sqrt{\lambda}$, then $w$ still lies on the unit circle and is not real, but $w+1/w$ is a real number.
\end{rem}

%%%%%%%%%%%%%%%%%%%%
\begin{proof}[Proof of Lemma \ref{poly_azumaya_positive}]
It is elementary to check that the polynomials $f_a(t)$ are irreducible for $a\geq 7$ (note that the polynomial is reducible when $a=6$). Furthermore, using sign changes of $f_a(t)$ evaluated at $0$, $1/2$, $1$, $2$ and $a$, the Intermediate Value Theorem shows that for $a\geq 7$, $f_a(t)=0$ has solutions in the intervals $[0,1/2]$, $[1/2,1]$, $[1,2]$ and $[2,a]$. Thus $f_a$ has four positive real roots.

We now check that for $a=k^2+2$, $k\geq 3$, that $f_a(t)$ is Azumaya positive. Set $w^2=t$, then note that $f_{k^2+2}(t)$ factors as
\[
-(-1 - kw + w^2 + kw^3 - w^4)(1 - kw - w^2 + kw^3 + w^4).
\]
Therefore $\Q(w)=\Q(t)$. It can be shown directly that the minimal polynomial for $w+1/w$ over $\Q$ is $t^4-(6+k^2)t^2+(9+4k^2)$, i.e., $\Q(w)=\Q(w+1/w)$ as required. Note that $f_{k^2+2}(t)$ is reducible for $k=2$, and when $k = 1$ the given polynomial for $w + 1/w$ is reducible, hence our assumption that $k \ge 3$.

\medskip

For the second part, suppose that $m = 2 \ell$ is even. The roots of $p_m(t)$ are
\[
z = \frac{(2 \ell + 1) \pm \sqrt{4 \ell + 1}}{2 \ell}.
\]
As before, let $w$ be a square root of $z$. We first notice that, as with $q_m(t)$, $\Q(z + z^{-1})$ is again $\Q$. Indeed,
\[
z + z^{-1} = \frac{2 \ell + 1}{\ell} \in \Q.
\]
Recall from the case when $m$ is odd that we have $\Q(z + z^{-1}) = \Q$ and $\Q(w + w^{-1}) / \Q(z + z^{-1})$ is either degree one or two, as is the extension $\Q(w) / \Q(w + w^{-1})$.

First, suppose that $\Q(w)$ is quartic over $\Q$. Then $\Q(w + w^{-1})$ is at most quadratic over $\Q$, and it follows that $\Q(w) / \Q(w + w^{-1})$ must be quadratic. In particular, the two fields are not equal and so the polynomial is Azumaya negative.

Now, we consider the opposite extreme, where the roots of $p_m(t)$ are rational. This occurs if and only if $4 \ell + 1$ is a square, and it is easy to check that $4 \ell + 1$ is a square if and only if $\ell$ is the product of two consecutive integers. If $\ell = q(q + 1)$, then
\[
z \in \left\{ \frac{q+1}{q}, \frac{q}{q + 1} \right\}.
\]
We claim that $\Q(w) / \Q$ is quadratic. Notice that $q$ and $q + 1$ are coprime, so $z$ is given as a fraction in reduced form. Then, $w \in \Q$ if and only if $q$ and $q + 1$ are both squares, which is impossible. Then $w + w^{-1} = (z + 1) / w$ clearly cannot be a rational number, else $w$ would be rational, so $\Q(w) = \Q(w + w^{-1})$.

Finally, suppose that $\Q(z)$ is quadratic over $\Q$ and $\Q(w) = \Q(z)$. Note that $4 \ell + 1$ is not a square. In other words, suppose that
\[
z = \frac{(2 \ell + 1) \pm \sqrt{4 \ell + 1}}{2 \ell} = \left( a + b \sqrt{4 \ell + 1} \right)^2
\]
for $a, b \in \Q$. This happens if and only if:
\begin{align*}
b &= \pm \frac{1}{4 a \ell} \\
a^2 &\in \left\{ \frac{1}{4 \ell}, \frac{4 \ell + 1}{4 \ell} \right\}
\end{align*}
However, $(4 \ell + 1) / 4 \ell$ is not a square of a rational number. Indeed, the numerator and denominator are coprime and our assumption that $z$ is quadratic over $\Q$ implies that $4 \ell + 1$ is not a rational square. It follows that $a^2 = 1/4 \ell$, so $\ell = q^2$ is necessarily a square and
\[
w = \pm \frac{1}{2 q}(1 \pm \sqrt{4 \ell + 1}) \in \Q(z)
\]
are the square roots of $z$. Then
\[
\left( \frac{1 + \sqrt{4 \ell + 1}}{2 q} \right)^{-1} = \frac{-1 + \sqrt{4 \ell + 1}}{2 q},
\]
so it follows that
\[
w + w^{-1} = \pm \frac{1}{q} \sqrt{4 \ell + 1} \notin \Q.
\]
We then have that $\Q(w) = \Q(w + w^{-1})$.

In summary, we showed that $\Q(w) = \Q(w + w^{-1})$ for $m=2 \ell$ where $\ell$ is either a square or the product of two consecutive integers. This completes the proof of the second case. The third part can be handled by direct computation.
\end{proof}
%%%%%%%%%%%%%%%%%%%%

%%%%%%%%%%%%%%%%%%%%
\subsection{Applications to Alexander polynomials}
%%%%%%%%%%%%%%%%%%%%

We now discuss applications of the results in \S \ref{sec:AZpolys} to hyperbolic knot complements and prove Theorem \ref{example_conseq}.

\medskip

\noindent
\textbf{Twist knots:}

\medskip

Let $T_m$ be the twist knot with $m\geq 1$ half-twists. Other than the trefoil (i.e., $m = 1$), $T_m$ is always a hyperbolic knot. Then
\[
\Del_{T_m}(t) = \begin{cases} q_m(t)& m\ \textrm{odd} \\
p_m(t) & m\ \textrm{even} \end{cases}
\]
See \cite{Ro}. Note that the case $m = 2$ is the figure-eight knot, and this is the only fibered hyperbolic twist knot.

As noted previously, \cite{MPV} shows that for a hyperbolic twist knot $T_m$, the canonical component is the unique component containing the character of an irreducible representation. Therefore it has field of constants $\Q$ by Lemma \ref{lem:UniqueToQ} (this is also clear from \cite{MPV}). Since each root of the Alexander polynomial has multiplicity one, the second part of Theorem \ref{example_conseq} now follows directly from Lemmas \ref{poly_azumaya_positive} and
\ref{poly_azumaya_negative}.

\medskip

\noindent
\textbf{Infinitely many fibered hyperbolic knots that are Azumaya positive:}

\medskip

As remarked earlier, any monic reciprocal polynomial is the Alexander polynomial of a fibered hyperbolic knot. In particular, for the polynomials $f_a(t)$ (with $a=k^2+2$) and $f(t)$ of Lemma \ref{poly_azumaya_positive}, $f_a(t)$ and $f(t)$ are the Alexander polynomials of a fibered hyperbolic knot, and so these knots will have canonical components that are Azumaya positive. The construction \cite{St} gives a method to produce arborescent knots with the given Alexander polynomial; we will not reproduce this here. However, we do note that by \cite{J} these knots cannot be alternating. These are the knots $K_n$ ($n=k^2+2$) referred to in Theorem \ref{example_conseq}(1).

As noted in the remark following the statement of Lemma \ref{poly_azumaya_positive}, $f_7(t)$ is also Azumaya positive. This polynomial is known to be the Alexander polynomial of the knot $8_{12}$ (see \cite{Ro}), which has hyperbolic volume $8.935856928\ldots$ and is the $2$-bridge knot with normal form $(29/12)$. Using Mathematica, it can be shown that the canonical component in this case is the unique component of the character variety containing the character of an irreducible representation. This computation produces a plane curve of total degree $22$, and using the algebraic packages in Magma \cite{Mag}, one can compute that the genus of the smooth model is $20$.

\medskip

\noindent
\textbf{Infinitely many fibered hyperbolic knots that are Azumaya negative:}

\medskip

Arguing as above using \cite{St} with the irreducible polynomials $p_\lambda(t)$ of Lemma \ref{poly_azumaya_negative} we can easily construct infinitely many fibered hyperbolic knots whose Alexander polynomials are Azumaya negative. It remains to ensure that this condition fails for the canonical component. To arrange this, we will use the family of $(-2,3,n)$-pretzel knots where $n\geq 7$ is odd and not divisible by $3$. The following will complete the proof of the existence of infinitely many fibered knots that are Azumaya negative.

%%%%%%%%%%%%%%%%%%%%
\begin{prop}
\label{pretzel}
Let $\mathcal{K}_n$ be the $(-2,3,n)$-pretzel knot where $n\geq 7$ is odd and not divisible by $3$. Then:
\begin{enumerate}

\item $\mathcal{K}_n$ is a fibered knot;

\item $X_0(\mathcal{K}_n)$ is the unique component of the character variety containing the character of an irreducible representation, hence it has field of constants $\Q$;

\item the Alexander polynomial $\Delta_{\mathcal{K}_n}(t)$ is of the form $p_\lambda(-t)$ for some Salem number $\lambda$.

\end{enumerate}
\end{prop}
%%%%%%%%%%%%%%%%%%%%

%%%%%%%%%%%%%%%%%%%%
\begin{proof}
The first part follows from work of Gabai, \cite{Ga}. The second part follows directly from Theorem 1.6 of \cite{Mat} and Lemma \ref{lem:UniqueToQ}. The Alexander polynomials of these pretzel knots is computed for example in \cite{EHi}, where the polynomials in question are described as
\[
\frac{1+2t+t^4+t^{1+r}-t^3-t^{3+r}+t^{5}+t^{2+r}+2t^{5+r}+t^{6+r}}{(1+t)^3},
\]
which for $r$ odd simplifies to
\[
P_r(t)=t^{3+r}-t^{2+r}+t^r-t^{r-1}+t^{r-2}-\ldots -t^4+t^3-t+1.
\]
It is proved in \cite{EHi} (using \cite{FP}) that $P_r(-t)$ has a Salem number as a root.
\end{proof}
%%%%%%%%%%%%%%%%%%%%

For the sake of concreteness, we provide some additional details for the case of the $(-2,3,7)$-pretzel knot, $\mathcal{K}_7$ in the above notation. This knot is fibered of genus $5$ with $\Delta_{\mathcal{K}_7}(t)=L(-t)$ where $L(t)$ is the famous Lehmer polynomial, the irreducible polynomial of the Salem number of conjectured minimal Mahler measure $> 1$.

\medskip

\noindent
\textbf{The canonical component $X_0(\mathcal{K}_7)$:}

\medskip

Using SnapPy \cite{SnapPy}, it can be shown that
\[
\Gamma=\pi_1(S^3\ssm \mathcal{K}_7) = <a,b\ |\ aab^{-1}aabbabb>.
\]
Since we are considering only irreducible representations, we can conjugate in $\SL_2(\C)$ so that
\begin{align*}
\rho(a) &= \begin{pmatrix}x & 1 \\ 0 & 1/x\end{pmatrix} \\
\rho(b) &= \begin{pmatrix}y & 0 \\ r & 1/y\end{pmatrix}
\end{align*}
for $r \neq 0$.

Using Mathematica, it is easy to compute the canonical component by evaluating $\rho$ on the relation, and converting to traces with coordinates
\begin{align*}
P&=\chi_\rho(a), \\
Q&=\chi_\rho(b) \\
R&=\chi_\rho(ab)
\end{align*}
we find that
\begin{align*}
P&= \frac{Q}{(Q^2-1)} \\
R&=\frac{(1-2Q^2)}{Q^2(Q^2-1)}.
\end{align*}
That is, $\wt{X_0(\mathcal{K}_7)}$ is a rational curve with field of constants $\Q$.

Then $\wt{X_0(\mathcal{K}_7)}$ is Azumaya negative. As described in the introduction, in the light of Theorem \ref{thm:FirstMain} one can perhaps suspect this on experimenting with Snap \cite{snap}. Namely one finds Dehn surgeries on $\mathcal{K}_7$ that are hyperbolic and have invariant quaternion algebras with finite places of very different residue field characteristics in the ramification sets. For example, we find places associated with primes of residue characteristic $3$, $5$, $13$, $31$, $149$, $211$, $487$, $563$, and $34543$.

From Theorem \ref{thm:FirstMain}, we conclude that there are points $\chi_\rho \in C(\bar{\Q})$ for which $A_\rho$ has ramification at a prime of residue characteristic $\ell$ for $\ell$ ranging over a set of rational primes with positive Dirichlet density. In this paper we only claim that this set is infinite, but Harari \cite{Harari} furthermore argues that this set of primes has positive density. However, we cannot conclude that these $\chi_\rho$ are the characters of hyperbolic Dehn surgeries on $\mathcal{K}_7$, though experiment suggests this may indeed be the case. This indicates that much more fruit can be borne of a better understanding of the arithmetic distribution on $C(\bar{\Q})$ of the characters of Dehn surgeries on hyperbolic knots. %%AR: added some stuff

To that end we propose the following:

\begin{conj}\label{conj:inf_bad}
Let $K$ be a hyperbolic knot in $S^3$ for which the canonical component is Azumaya negative, and let $S$ be the infinite set of rational primes $p$ provided by Theorem \ref{thm:FirstMain}(3). Then for $p \in S$, there exists a hyperbolic Dehn surgery $N$ of $K$, and 
a prime $\mathcal{P}$ of the trace field $k_N$ such that the quaternion algebra $A_N$ is ramified at $\mathcal{P}$.\end{conj}

The first evidence towards this conjecture is provided in recent work of N.\ Rouse \cite{Rou}.

%%%%%%%%%%%%%%%%%%%%
\subsection{$L$-space knots}
%%%%%%%%%%%%%%%%%%%%

We now make some comments on an apparent connection between our conditions of Azumaya positive and negative, and a collection of knots that have been of interest through Heegaard--Floer homology, so-called {\em $L$-space knots}.

Following Ozsvath--Szabo \cite{OS}, an \emph{L-space} is a rational homology $3$-sphere $M$ for which its Heegaard--Floer homology $\widehat{HF}(M)$ is as simple as possible, i.e., is a free abelian group of rank equal to $|H_1(M;\Z)|$. Examples of $L$-spaces are lens spaces (excluding $S^2\times S^1$), other 3-manifolds covered by $S^3$, as well as many Seifert fibered spaces and hyperbolic manifolds. A knot $K\subset S^3$ is called an \emph{$L$-space knot} if $S^3\ssm K$ admits a (positive) Dehn surgery giving an $L$-space.

Examples of $L$-space knots are the $(-2,3,n)$-pretzel knots $\mathcal{K}_n$ (see \cite{OS, LM}), which are Azumaya negative by Proposition \ref{pretzel}. These along with the torus knots $T(2,2n+1)$, which are not hyperbolic, are the only $L$-space Montesinos knots \cite{BM, LM}. An important result of Ni \cite{Ni} shows that $L$-space knots are fibered, and in the context of this paper we have.

%%%%%%%%%%%%%%%%%%%%
\begin{prop}
\label{Lspace}
No Azumaya positive fibered knot in Theorem \ref{example_conseq} is an $L$-space knot.
\end{prop}
%%%%%%%%%%%%%%%%%%%%

%%%%%%%%%%%%%%%%%%%%
\begin{proof}
The knots $K_n$ in Theorem \ref{example_conseq}(1) have Alexander polynomials $f_{k^2+2}(t)$ (taking $n=k^2+2$). In particular these have non-zero coefficients different from $\pm 1$. However, if $K$ is an $L$-space knot, Ozsvath--Szabo \cite{OS} proved that the symmetrized Alexander polynomial of $K$ is of the form
\[
(-1)^k + \sum_{j = 1}^k (-1)^{k - j}\left(t^{n_j} + t^{-n_j}\right),
\]
and so all non-zero coefficients of $\Delta_K(t)$ are $\pm 1$. 

Similarly, the Alexander polynomial condition of \cite{OS} applies to show that a fibered knot $K$ (as in Lemma \ref{poly_azumaya_positive}) with Alexander polynomial 
\[
\Delta_K(t)=t^8-3t^7 +5t^6-7t^5+9t^4-7t^3+5t^2-3t+1
\]
cannot be an $L$-space knot.\end{proof}

Based on this we make the following conjecture.

%%%%%%%%%%%%%%%%%%%%
\begin{conj}\label{conj:Lspace}
Let $K$ be a (fibered) hyperbolic knot for which the canonical component is Azumaya positive. Then $K$ is not an $L$-space knot.
\end{conj}
%%%%%%%%%%%%%%%%%%%%

%%%%%%%%%%%%%%%%%%%%
\begin{rem}
One can view Conjecture \ref{conj:Lspace} as another instance of Azumaya positivity placing significant restrictions on the possible Dehn surgeries on a hyperbolic knot. Indeed, just as Theorem \ref{thm:MainDehn} places severe restrictions on the arithmetic invariants of hyperbolic Dehn surgeries on Azumaya positive knots, this conjecture implies that an Azumaya positive knot is excluded from having certain rational homology $3$-spheres arising from a Dehn surgery. That our results are entirely determined by arithmetic properties of the Alexander polynomial, and that results of Ozsvath--Szabo have Alexander polynomial ties of a very similar flavor, lead us to believe that such a connection should exist.
\end{rem}
%%%%%%%%%%%%%%%%%%%%

We now give some evidence for this conjecture. The starting point is the following:

%%%%%%%%%%%%%%%%%%%%

\begin{prop}\label{warmupLspace}
Let $K$ be a hyperbolic knot and $C \subset X(K)$ the canonical component. Suppose that $C$ has field of constants $\Q$ and contains the character of a non-abelian reducible representation associated with a simple root of $\Del_K(t)$ on the unit circle. Then $A_{k(C)}$ does not extend to an Azumaya algebra over the smooth projective model of $C$.
\end{prop}
%%%%%%%%%%%%%%%%%%%%

%%%%%%%%%%%%%%%%%%%%
\begin{proof}
This follows immediately from Theorem \ref{thm:BigConverse} and Remark \ref{unitcircle}.
\end{proof}
%%%%%%%%%%%%%%%%%%%%

The relevance of this to $L$-space knots is the observation of Culler and Dunfield \cite{CullerDunfield} that if $K$ is an $L$-space knot, one can apply the results of \cite{Konvalina--Matache} to see that $\Del_K(t)$ has a root on the unit circle. 

%%%%%%%%%%%%%%%%%%%%
\begin{cor}
\label{notLspace}
Suppose that $K$ is a hyperbolic $L$-space knot for which the canonical component $C$ is the unique component of $X(K)$ containing the character of an irreducible representation. Then $K$ is Azumaya negative.
\end{cor}
%%%%%%%%%%%%%%%%%%%%

%%%%%%%%%%%%%%%%%%%%
\begin{proof}
Given the observation of Culler and Dunfield above, and the hypothesis
on $C$, the corollary follows from Proposition \ref{warmupLspace} and Lemma \ref{lem:UniqueToQ}.
\end{proof}
%%%%%%%%%%%%%%%%%%%%

Finally, one other piece of experimental evidence to support Conjecture \ref{conj:Lspace} is that Culler and Dunfield \cite{CullerDunfield} also remark that they were unable to find an $L$-space knot whose Alexander polynomial does not have a simple root on the unit circle.

%%%%%%%%%%%%%%%%%%%%
\subsection{Bi-orderability}
%%%%%%%%%%%%%%%%%%%%

A group $G$ is \emph{left-orderable} if there is a strict total ordering $<$ of its elements that is invariant under multiplication on the left: $g < h$ implies $fg < fh$ for $f,g,h \in G$. It is easy to see that $G$ is left-orderable if and only if it is right-orderable. An ordering of $G$ that is invariant under multiplication on both sides will be called a bi-ordering. If such an ordering exists we say that $G$ is \emph{bi-orderable}.

It is well-known that all knot groups are left-orderable (since they are locally indicable, \cite{BoyRW}). However, admitting a bi-order is more subtle. It was shown by Perron and Rolfsen \cite{PR} that if a fibered knot $K$ has the property that all roots of $\Delta_K(t)$ are real and positive, then the knot group is bi-orderable. Clay and Rolfsen proved a partial converse \cite{CR}: If $K$ is a non-trivial fibered knot in $S^3$ with bi-orderable fundamental group, then $\Delta_K(t)$ has at least one root that is real and positive (in fact, it has at least two).

Orderability has recently seen connections to various aspects of the topology of 3-manifolds, one compelling example being that it appears that Heegaard--Floer homology is connected with left-orderability of the fundamental group of a closed 3-manifold. Another instance of this connection between $L$-spaces and orderability is provided by the following result of Clay and Rolfsen \cite[Thm.\ 1.2]{CR}: \emph{If $K\subset S^3$ is a non-trivial knot and $\pi_1(S^3\ssm K)$ is bi-orderable, then $K$ is not an $L$-space knot.} In the context of this paper we have a ``bi-ordered analogue'' of Proposition \ref{Lspace}.

%%%%%%%%%%%%%%%%%%%%
\begin{prop}
\label{biorder}
Every Azumaya positive fibered knot in Theorem \ref{example_conseq}(1) has bi-orderable fundamental group.
\end{prop}
%%%%%%%%%%%%%%%%%%%%

%%%%%%%%%%%%%%%%%%%%
\begin{proof}
The knots $K_n$ of Theorem \ref{example_conseq}(1) have Alexander polynomials $f_{k^2+2}(t)$ (taking $n=k^2+2$). From Lemma \ref{poly_azumaya_positive}(1), all roots of these polynomials are real and positive. Hence the work of Clay and Rolfsen \cite{CR} described above implies that the knot groups $\pi_1(S^3\ssm K_n)$ (or indeed knot groups of all knots with Alexander polynomial $f_{k^2+2}(t)$) are bi-orderable.
\end{proof}
%%%%%%%%%%%%%%%%%%%%

%%%%%%%%%%%%%%%%%%%%
\begin{rem}
Note also that in the remark following the statement of Lemma \ref{poly_azumaya_positive} we observed $f_7(t)$ is Azumaya positive. This also has all roots real and positive, and so if $K$ is any fibered knot with $\Delta_K(t)=f_7(t)$ then we once again have from \cite{CR} that any associated knot group is bi-orderable. A particular example of such a knot is the knot $8_{12}$ mentioned above.
\end{rem}
%%%%%%%%%%%%%%%%%%%%

%%%%%%%%%%%%%%%%%%%%
\begin{rem}
Note that the polynomial of Lemma \ref{poly_azumaya_positive}(3) has all roots imaginary and so by \cite{CR}, any knot $K$ which has this as its Alexander polynomial has knot group that is not bi-orderable, even though the knot will be Azumaya positive.
\end{rem}
%%%%%%%%%%%%%%%%%%%%

%%%%%%%%%%%%%%%%%%%%
\subsection{An example with trivial Alexander polynomial}
%%%%%%%%%%%%%%%%%%%%

We have been unable to find a knot with trivial Alexander polynomial whose canonical component could be computed explicitly for us to record. However, given Remark \ref{notS3}, we can content ourselves with an example of a knot in $S^2\times S^1$ that can be analyzed.

The example in question is a manifold from the census of cusped hyperbolic $3$-manifolds that can be built from at most $5$ tetrahedra. In the original version of SnapPy, it is the manifold denoted $m137$. This manifold appeared in \cite{Dun} and more recently in \cite{Gao}, which points out that the Alexander polynomial is $1$.

It is also shown in \cite{Gao} that the canonical component $X_0$ is a curve $C$ in ${\C}^2$ with field of constants $\Q$ as the vanishing locus of the polynomial:
\[
p(s,t)=(-2-3s+s^2)t^4 + (4+4s-s^2-s^3)t^2-1,
\]
where $s$ and $t$ are certain trace functions. Using Magma \cite{Mag} it can be shown that the genus of $\wt{C}$ is $5$ and that the curve is not hyperelliptic.

It is shown in \cite{Gao} that there are both characters of irreducible $\SL_2(\R)$ representations and $\SU(2)$ representations on $C$. Indeed there are $6$ connected components of real characters in total, two of which correspond to $\SU(2)$ representations. If $\mathcal{A}$ denotes the Azumaya algebra over $C$ we deduce from Theorem \ref{non-trivial}(2) that the class $\beta(\mathcal{A})$ does not lie in the image of the Tate--Shafarevich group of the Jacobian of the smooth projective model $\wt{C}$ (we were not able to check whether or not $\Sha$ is trivial in this case).

Experimenting with Snap \cite{snap} (as done in \S 1), one sees only ramification at the real places. In particular, we see that $[\mathcal{A}]$ is indeed a non-trivial element of $\Br(\wt{C})$.

%%%%%%%%%%%%%%%%%%%%
\section{The figure-eight knot}\label{sec:Fig8}
%%%%%%%%%%%%%%%%%%%%

Recall that the second part of Theorem \ref{example_conseq} shows that the figure-eight knot is Azumaya positive, which gives Theorem \ref{thm:Fig8}(1). In this section we work out in detail which Azumaya algebra occurs (i.e., the element of the Brauer group of the canonical curve) and use this to prove Theorem \ref{thm:Fig8}.

We begin by recalling the computation of $X_0$ in this case. To that end, let $\Gam$ be the figure-eight knot group. Then $\Gam$ has presentation:
\begin{equation}\label{eq:Fig8Pres}
\Gam = \langle a, b\ |\ b = w a w^{-1} \rangle \quad\quad w = [a^{-1}, b] = a^{-1} b a b^{-1}
\end{equation}
As in the previous section, we can use computer algebra software like Mathematica to compute a polynomial whose vanishing set defines the canonical component. In this case $X_0$ is described in the affine plane $\C[T, R]$ as:
\begin{equation}
R T^2 - 2 T^2 - R^2 + R + 1 = 0 \label{eq:Fig8aff}
\end{equation}
\begin{align*}
T &= \chi_\rho(a) \\
R &= \chi_\rho(a b)
\end{align*}
Note that $a$ and $b$ are conjugate, so we also have $T = \chi_\rho(b)$. We can also change to the affine plane $\C[y, z]$, where $y = T(R - 2)$ and $z = R - 1$ and obtain the Weierstrass form
\begin{equation}\label{eq:Fig8Weier}
y^2 = z^3 - 2 z + 1.
\end{equation}
In particular, since this curve is a non-singular plane cubic we deduce that there is a unique component containing the characters of absolutely irreducible representations, and it hence coincides with $X_0$. We start by proving part (2) of Theorem \ref{thm:Fig8}.

%%%%%%%%%%%%%%%%%%%%
\begin{proof}[Proof of Theorem \ref{thm:Fig8}(3)]
Recall that, up to scaling, the figure-eight knot has Alexander polynomial
\[
\Delta_K(t) = t^2 - 3 t + 1.
\]
The roots over $\Q$ are
\[
z = \frac{3 \pm \sqrt{5}}{2},
\]
and $z = (\pm w)^2$, where
\[
w = \frac{1 \pm \sqrt{5}}{2}.
\]
Then $w + 1/w = \pm \sqrt{5}$, and we see explicitly that the figure-eight knot is Azumaya positive.

Now let $\ell$ be a prime, and let $\Delta_{K, \ell}(t)$ denote $\Delta_K(t)$ considered as an element of $\F_\ell[t]$. The argument over $\Q$ applies \emph{mutatis mutandis} to show that $(\star_\ell)$ of Theorem \ref{thm:MainKnotIntegral} holds when $\ell \not \in \{2,5\}$. When $\ell = 5$, $x = -1 \in \mathbb{Z}/5$ is the only root of $\Delta_{K, 5}(t)$, and $-1$ has square roots $w = 2, 3 \in \F_5$, so $(\star_5)$ holds.

Finally, when $\ell = 2$ we have
\[
\Delta_{K, 2}(t) = t^2 + t + 1
\]
so if $z$ is a root of $\Delta_{K, 2}(t)$, $\F_2(z) \cong \F_4$ and $z^3 = 1$, hence $\F_2(z)$ contains the square root $w = 1/z$ of $z$. However,
\[
w + 1/w = 1 \in \F_2,
\]
so $\F_2(w) \neq \F_2(w + 1/w)$, and hence $(\star_2)$ fails. This proves that $S = \{2\}$ is the minimal set for which the conditions of Theorem \ref{thm:MainKnotIntegral} holds, and this completes the proof of Theorem \ref{thm:Fig8}(3).
\end{proof}
%%%%%%%%%%%%%%%%%%%%

We now explore the Azumaya algebra $A_E$ over the smooth projective model of $E$, along with the algebra $A_{k(E)}$ over the function field $k(E)$ of $E$, in more detail. We will use two affine patches of $E$. The first is the affine curve $E_0$ defined by \eqref{eq:Fig8Weier} in the $(y,z)$ plane. The second is the affine curve $E_0^\prime=X_0$ defined by \eqref{eq:Fig8aff} in the $(T, R)$-plane.

We begin by giving Hilbert symbols in our various coordinates.

%%%%%%%%%%%%%%%%%%%%
\begin{lem}\label{lem:fig8Hilbert}
Over the function field $k(E)$ of $E$, we have Hilbert symbols
\[
A_{k(E)} = \left( \frac{T^2 - 4, R - 3}{k(E)} \right) = \left( \frac{z^3 - 4 z^2 + 6 z - 3, z - 2}{k(E)} \right).
\]
The specialization of $A_{k(E)}$ over $(y,z) = (\pm 1,0)$ is a division algebra over $\mathbb{Q}$, so 
$A_{k(E)}$ (resp.\ $A_E$) is a non-trivial Azumaya algebra over $k(E)$ (resp.\ $E$).
\end{lem}
%%%%%%%%%%%%%%%%%%%%

%%%%%%%%%%%%%%%%%%%%
\begin{proof}
For each point $(T, R)$ on $E_0^\prime$ we define:
\begin{align*}
\alpha &= \chi_\rho(a)^2 - 4 \\
\beta &= \chi_\rho([a,b]) - 2
\end{align*}
Since $\Gam$ is generated by $a$ and $b$, it is clear from the methods used in this paper that we can use $\alpha, \beta$ to define a Hilbert symbol for $A_E$ over the function field $k(E)$ of $E$.

A standard identity for traces of $2 \times 2$ matrices is
\[
\tr([A, B]) = \tr(A)^2 + \tr(B)^2 + \tr(A B)^2 - \tr(A) \tr(B) \tr(A B) - 2,
\]
which, at a point $(T, R)$ on $E_0^\prime$, allows us to define
\begin{align*}
\alpha &= \alpha(T, R) \\
&= T^2 - 4 \\
\beta &= \beta(T, R) \\
&= 2 T^2 + R^2 - R T^2 - 4 \\
&= R - 3
\end{align*}
where the last equality comes from \eqref{eq:Fig8aff}.

We also have
\begin{align*}
\alpha^\prime &= \alpha (R - 2)^2 \\
&= z^3 - 4 z^2 + 6 z - 3
\end{align*}
and notice that $\beta = z - 2$. Then $\{\alpha^\prime, \beta\}$ also gives a Hilbert symbol for $A_{k(E)}$. The last statement of the lemma follows from that fact that specializing at $z = 0$ gives the quaternion algebra over $\mathbb{Q}$ with Hilbert symbol
\[
\left( \frac{ - 3, - 2}{\Q} \right ).
\]
This algebra ramifies over the real place and hence is non-trivial. The lemma follows.
\end{proof}
%%%%%%%%%%%%%%%%%%%%

We now give a minimal extension of $k(E)$ that splits $A_E$.

%%%%%%%%%%%%%%%%%%%%
\begin{lem}\label{lem:Fig8QiSplits}
The algebra $A_E$ splits over $k(E)(i)$. In other words, $A_E \otimes_{k(E)} k(E)(i)$ is isomorphic to the $2 \times 2$ matrix algebra over $k(E)(i)$. Consequently, given any point $p \in E$ with associated quaternion algebra $A_p$ over the residue field $k(p)$ of $p$, we have
\[
A_p \otimes_{k(p)} k(p)(i) \cong \M_2(k(p)(i)).
\]
\end{lem}
%%%%%%%%%%%%%%%%%%%%

%%%%%%%%%%%%%%%%%%%%
\begin{proof}
It suffices to split $A_{k(E)}$ over $k(E)(i)$. In the function field $k(E)(i) = k(E) \otimes_\mathbb{Q} \mathbb{Q}(i)$ let
\begin{align*}
m_1 &= -2 - i + (1 + i) z \\
m_2 &= i \\
m_3 &= 1 - i - z \\
r &= m_1 + m_2 I + m_3 J
\end{align*}
where $I^2 = \alpha^\prime$ and $J^2 = \beta$, with $\alpha^\prime$ and $\beta$ as in the proof of Lemma \ref{lem:fig8Hilbert}. Then $r$ has reduced norm
\[
m_1^2 + \alpha^\prime - m_3^2 \beta = 0,
\]
so $A_{k(E)} \otimes_{k(E)} k(E)(i)$ must split. This proves the lemma.
\end{proof}
%%%%%%%%%%%%%%%%%%%%

We now consider the behavior of $A_E$ at an ideal point. Using homogeneous coordinates $[W:Y:Z]$, so $y = Y/W$ and $z = Z/W$, it is clear that $E_0$ has a single ideal point $p_\infty$ at $[0:1:0]$. Let $E = E_0 \cup \{p_\infty\}$. The elements
\begin{align*}
\alpha_\infty &= \frac{1}{y^2} \alpha^\prime = \frac{z^3 - 4 z^2 + 6 z - 3}{z^3 - 2 z + 1} \\
\beta_\infty &= \frac{z^2}{y^2} \beta = \frac{z^3 - 2 z^2}{z^3 - 2 z + 1}
\end{align*}
give a well-defined Hilbert symbol for $A_E$ which specializes at $p_\infty$ to
\[
\left( \frac{1, 1}{\Q} \right) \cong \M_2(\Q).
\]
This shows:

%%%%%%%%%%%%%%%%%%%%
\begin{lem}\label{lem:Fig8IdealSplits}
The $\Q$-quaternion algebra $A_\infty$ given by specialization of $A_E$ at the ideal point $p_\infty$ splits.
\end{lem}
%%%%%%%%%%%%%%%%%%%%

We now complete the proof of Theorem \ref{thm:Fig8}.

%%%%%%%%%%%%%%%%%%%%
\begin{proof}[Proof of Theorem \ref{thm:Fig8}(2)]
The class $[A_E]$ of $A_E$ in the Brauer group $\Br(E)$ is non-trivial by Lemma \ref{lem:fig8Hilbert}. Lemma \ref{lem:Fig8IdealSplits} then implies that $[A_E]$ lies in the kernel $\Br^0(E)$ of the specialization map of Brauer groups $\Br(E) \to \Br(\Q)$ associated with specialization at $p_\infty$.

The argument of \cite[Lem.\ 2.1]{Wittenberg} shows $\Br^0(E) $ is isomorphic to the Galois cohomology group $H^1(\Q,E(\overline{\Q}))$, since the hypothesis in \cite{Wittenberg} that the $2$-torsion $E[2]$ of $E(\overline{\Q})$ is defined over $\Q$ is not needed for this conclusion. Lemma \ref{lem:Fig8QiSplits} shows $[A_E]$ is in the kernel of the homomorphism
\[
\Br^0(E) \to \Br^0(E \otimes_{\mathbb{Q}} \mathbb{Q}(i))
\]
induced by tensoring over $\Q$ with $\Q(i)$. Therefore $[A_E]$ is identified with an element of
\begin{align*}
&H^1(\mathrm{Gal}(\mathbb{Q}(i)/\mathbb{Q}),E(\mathbb{Q}(i))) \\
=\,\, &\mathrm{Ker}\{ H^1(\mathbb{Q},E(\overline{\mathbb{Q}})) \to H^1(\mathbb{Q}(i),E(\overline{\mathbb{Q}}))\}.
\end{align*}
The latter equality is a consequence of the restriction-inflation sequence in group cohomology; see \cite[Chap.\ VII.6, Chap.\ X]{Serre}.

The points $E(\mathbb{Q}(i))$ contain the (finite index) subgroup generated by the subgroup $E(\mathbb{Q})$ of points fixed by $\mathrm{Gal}(\mathbb{Q}(i)/\mathbb{Q})$ along with the subgroup of those points sent to their negatives by complex conjugation. The latter points correspond to rational points on the quadratic twist
\[
\tilde{E}: -y^2 = z^3 - 2z + 1.
\]
The curves $E$ and $\tilde{E}$ are modular of conductors $40$ and $80$, respectively, and they each have rank $0$ over $\mathbb{Q}$, which one can easily check in Sage \cite{Sage}. It follows that $E(\mathbb{Q}(i))$ is finite. The $2$-torsion of $E$ over $\overline{\mathbb{Q}}$ in $(y, z)$ coordinates is
\[
E(\overline{\mathbb{Q}})[2] = \left\{ p_\infty\, ,\, (0, 1)\, ,\, \left(0, \frac{-1 + \sqrt{5}}{2} \right)\, ,\, \left(0, \frac{-1 - \sqrt{5}}{2} \right) \right\}.
\]
It follows that $E(\mathbb{Q}(i))[2]$ has order $2$, and so the $2$-Sylow subgroup of the finite group $E(\mathbb{Q}(i))$ is a cyclic $2$-group with an action of the group $\mathrm{Gal}(\mathbb{Q}(i)/\mathbb{Q})$ of order $2$.

Then, $H^1(\mathrm{Gal}(\mathbb{Q}(i)/\mathbb{Q}),E(\mathbb{Q}(i)))$ is isomorphic to
\[
H^{-1}(\mathrm{Gal}(\mathbb{Q}(i)/\mathbb{Q}),E(\mathbb{Q}(i))),
\]
and the latter is cyclic since $E(\mathbb{Q}(i))$ is cyclic (see \cite[VIII.4]{Serre}). Since these cohomology groups are annihilated by $\# \mathrm{Gal}(\mathbb{Q}(i)/\mathbb{Q}) = 2$ and Lemma \ref{lem:fig8Hilbert} showed that $[A_E]$ is non-trivial, we conclude that $[A_E]$ is the unique non-trivial element of the group $H^1(\mathrm{Gal}(\mathbb{Q}(i)/\mathbb{Q}),E(\mathbb{Q}(i)))$. In particular, $[A_E]$ is the unique non-trivial element in $\Br^0(E)$ that becomes trivial after tensoring over $\Q$ with $\mathbb{Q}(i)$. This completes the proof of Theorem \ref{thm:Fig8}(2).
\end{proof}
%%%%%%%%%%%%%%%%%%%%

%%%%%%%%%%%%%%%%%%%%
\bibliographystyle{plain}
\bibliography{Azumaya}
%%%%%%%%%%%%%%%%%%%%

%%%%%%%%%%%%%%%%%%%%
\end{document}